\documentclass[10pt,a4paper]{article}

\usepackage{fullpage,amssymb,amsmath,amsfonts,hyperref,amsthm,paralist,graphicx,color,float, mathtools,tikz}
\usepackage{here}
\usepackage[normalem]{ulem}
\usepackage{centernot}
\usepackage{ifthen,xkeyval, tikz, calc, graphicx}
\usepackage{xcolor}
\usepackage{caption, subcaption}
\usepackage[textsize=scriptsize]{todonotes}
\usepackage{dsfont}
\usepackage{mathrsfs}
\usepackage{comment}
\graphicspath{{Figures/}{Figures/ExpansionPictures}{Figures/MiniDiagrams/}}

\newtheorem{theorem}{Theorem}[section]
\newtheorem{lemma}[theorem]{Lemma}
\newtheorem{prop}[theorem]{Proposition}
\newtheorem{corollary}[theorem]{Corollary}
\newtheorem{observation}[theorem]{Observation}

\theoremstyle{definition}
\newtheorem{definition}[theorem]{Definition}

\newcommand{\al}[1]{\begin{align*}#1\end{align*}}

\newcommand{\algn}[1]{\begin{align}#1\end{align}}
\newcommand{\eqq}[1]{\begin{equation}#1\end{equation}}

\newcommand{\p}{\mathbb P}
\newcommand{\pp}{\mathbb P_p}
\newcommand{\E}{\mathbb E}
\newcommand{\R}{\mathbb R}

\newcommand{\Zd}{\mathbb Z^d}

\newcommand{\N}{\mathbb N}

\newcommand{\dd}{\, \mathrm{d}}
\newcommand{\C}{\mathscr {C}}
\newcommand{\thinn}[1]{\langle #1 \rangle}
\newcommand{\piv}[1]{\textsf {Piv}(#1)}
\newcommand{\taup}{\tau_p}
\newcommand{\taupf}{\tau_p^\bullet}
\newcommand{\taupo}{\tau_p^\circ}

\newcommand{\ttaup}{\gamma_p}
\newcommand{\conn}[3]{#1 \longleftrightarrow #2\textrm { in } #3}

\newcommand{\oconn}{\leftrightsquigarrow}

\newcommand{\jeq}{J}
\newcommand{\jek}{J_k}
\newcommand{\throughconn}[1]{\xleftrightarrow{\ #1 \ }}
\newcommand{\offconn}[3]{#1 \longleftrightarrow #2\textrm { off } #3}
\newcommand{\noffconn}[3]{#1 \centernot\longleftrightarrow #2\textrm { off } #3}

\newcommand{\taupk}{\tau_{p,k}}

\newcommand{\ftau}{\widehat\tau_p}
\newcommand{\ftaupk}{\widehat\tau_{p,k}}

\newcommand{\connf}{D}
\newcommand{\fjeq}{\widehat\jeq}
\newcommand{\fconnf}{\widehat\connf}
\newcommand{\greenslam}{G_\lambda}

\newcommand{\fgreenslam}{\widehat G_\lambda}
\newcommand{\fgmu}{\widehat G_{\lamp}}

\newcommand{\ulam}{\widehat U_{\lamp}}
\newcommand{\fspace}{(-\pi,\pi]^d}
\newcommand{\fpip}{\widehat\Pi_p}
\newcommand{\lamp}{\lambda_p}

\newcommand{\trip}{\triangle_p}
\newcommand{\tripo}{\triangle_p^\circ}
\newcommand{\tripf}{\triangle_p^{\bullet}}
\newcommand{\tripof}{\triangle_p^{\bullet\circ}}
\newcommand{\tripoff}{\triangle_p^{\bullet\bullet\circ}}

\newcommand{\e}{\text{e}}
\newcommand{\orig}{\mathbf{0}}
\renewcommand{\i}{\text{i}}

\newcommand{\col}[1]{#1}

\newcommand{\hak}{\widehat a(k)}
\newcommand{\hao}{\widehat a(0)}
\newcommand{\hnk}{\widehat N(k)}
\newcommand{\hfk}{\widehat F(k)}

\numberwithin{equation}{section}
\allowdisplaybreaks

\definecolor{darkorange}{RGB}{255,165,0}
\definecolor{altviolet}{RGB}{139,0,139}
\definecolor{turquoise}{RGB}{64,224,208}

\title{Critical site percolation in high dimension}
\author{Markus Heydenreich\footnote{Ludwig-Maximilians-Universit\"at M\"unchen, Mathematisches Institut, Theresienstr.\ 39, 80333 M\"unchen, Germany. \newline E-mail: m.heydenreich@lmu.de; matzke@math.lmu.de} 
		\and Kilian Matzke\footnotemark[1] }
\begin{document}
\maketitle

\begin{abstract}
We use the lace expansion to prove an infra-red bound for site percolation on the hypercubic lattice in high dimension. This implies the triangle condition and allows us to derive several critical exponents that characterize mean-field behavior in high dimensions. 
\end{abstract}



\section{Introduction} \label{sec:introduction}
\subsection{Site percolation on the hypercubic lattice}
\label{sec:defs}
We consider site percolation on the hypercubic lattice $\Zd$, where sites are independently \emph{occupied} with probability $p\in[0,1]$, and otherwise \emph{vacant}. 
More formally, for $p \in [0,1]$, we consider the probability space $(\Omega, \mathcal F, \pp)$, where $\Omega = \{0,1\}^{\Zd}$, the $\sigma$-algebra $\mathcal F$ is generated by the cylinder sets, and $\pp = \bigotimes_{x\in\Zd} \text{Ber}(p)$ is a product-Bernoulli measure. We call $\omega\in\Omega$ a configuration and say that a site $x\in\Zd$ is \emph{occupied} in $\omega$ if $\omega(x)=1$. If $\omega(x)=0$, we say that the site $x$ is \emph{vacant}. For convenience, we identify $\omega$ with the set of occupied sites $\{x\in\Zd: \omega(x)=1\}$.

Given a configuration $\omega$, we say that two points $x \neq y\in\Zd$ are \emph{connected} and write $x \longleftrightarrow y$ if there is an \emph{occupied path} between $x$ and $y$---that is, there are points $x=v_0, \ldots, v_k=y$ in $\Zd$ with $k\in\N_0 := \N\cup \{0\}$ such that $| v_i-v_{i-1}| = 1$ (with $|y| = \sum_{i=1}^{d} |y_i|$ the $1$-norm) for all $1 \leq i \leq k$, and $v_i\in\omega$ for $1 \leq i \leq k-1$ (i.e., all  \emph{inner} sites are occupied). 
Two neighbors are automatically connected (i.e.,~$\{x \longleftrightarrow y\}=\Omega$ for all $x,y$ with $|x-y|=1$). 
\col{Many authors prefer a different definition of connectivity by requiring both endpoints to be occupied as well. These two notions are closely related and we explain our choice in Section~\ref{sec:intro:literature}.}
Moreover, we adopt the convention that $\{x \longleftrightarrow x\}=\varnothing$, that is, $x$ is \emph{not} connected to itself.

We define the \emph{cluster} of $x$ to be $\C(x):= \{x\} \cup \{y \in \omega: x \longleftrightarrow y\}$. Note that apart from $x$ itself, points in $\C(x)$ need to be occupied. We also define the expected cluster size (or \emph{susceptibility})  $\chi(p) = \E_p[|\C(\orig)|]$, where for a set $A \subseteq \Zd$, we let $|A|$ denote the cardinality of $A$, and $\orig$ the origin in $\Zd$.

We define the \emph{two-point function} $\tau_p\colon \Zd\to[0,1]$ by
	\[\tau_p(x):=\pp(\orig \longleftrightarrow x).\]
The \emph{percolation probability} is defined as $\theta(p) := \pp(\orig \longleftrightarrow \infty) = \pp( | \C(\orig)| = \infty)$. We note that $p \mapsto \theta(p)$ is increasing and define the \emph{critical point} for $\theta$ as
	\[ p_c = p_c(\Zd) = \inf\{ p > 0: \theta(p) > 0 \}. \]
Note that we can define a critical point $p_c(G)$ for any graph $G$. As we only concern ourselves with $\Zd$, we write $p_c$ or $p_c(d)$ the refer to the critical point of $\Zd$.

\subsection{Main result} \label{sec:intro:main_theorems}
The \emph{triangle condition} is a versatile criterion for several critical exponents to exist and to take on their mean-field value. In order to introduce this condition, we define the \emph{open triangle diagram} as
	\[ \triangle_p(x) = p^2 (\tau_p \ast\tau_p\ast\tau_p)(x) \]
and the triangle diagram as $\triangle_p = \sup_{x \in \Zd} \triangle_p(x)$. In the above, the convolution `$\ast$' is defined as $(f\ast g)(x) = \sum_{y\in\Zd} f(y) g(x-y)$. We also set $f^{\ast j} =f^{\ast (j-1)} \ast f$ and $f^{\ast 1} \equiv f$. The triangle condition is the condition that $\triangle_{p_c} <\infty$. To state Theorem~\ref{thm:main_theorem_triangle_condition}, we recall that the discrete Fourier transform of an absolutely summable function $f:\Zd \to \R$ is defined as $\widehat f:\fspace \to \mathbb C$ with
	\[ \widehat f(k) = \sum_{x \in \Zd} \e^{\i k \cdot x} f(x),\]
where $k\cdot x = \sum_{j=1}^{d} k_j x_j$ denotes the scalar product. Letting $D(x) = \tfrac{1}{2d} \mathds 1_{\{|x|=1\}}$ for $x\in\Zd$ be the step distribution of simple random walk, we can formulate our main theorem:

\begin{theorem}[The triangle condition and the infra-red bound]\label{thm:main_theorem_triangle_condition}
There exist $d_0 \geq 6$ and a constant $C=C(d_0)$ such that, for all $d > d_0$,
	\eqq{ p|\ftau(k)| \leq \frac{|\fconnf(k)| + C /d}{1 - \fconnf(k)} \label{eq:intro:infra-red_bound}}
for all $k \in (-\pi, \pi]^d$ uniformly in $p \in [0,p_c]$ (we interpret the right-hand side of~\eqref{eq:intro:infra-red_bound} as $\infty$ for $k=0$). Additionally, $\triangle_p \leq C/d$ uniformly in $[0, p_c]$, and the triangle condition holds.
\end{theorem}

\subsection{Consequences of the infra-red bound}
The triangle condition is the classical criterion for mean-field behavior in percolation models. The triangle condition implies readily that $\theta(p_c)=0$ (since otherwise $\triangle_{p_c}$ could not be finite), a problem that is still open in smaller dimension (except $d=2$). 

Moreover, the triangle condition implies that a number of critical exponents take on their mean-field values. Indeed, using results by Aizenman and Newman~\cite[Section 7.7]{AizNew84}, the triangle condition implies that the critical exponent $\gamma$ exists and takes its mean-field value 1, that is 
	\eqq{ 	\frac{c}{p_c-p}\leq \chi(p)\leq \frac{C}{p_c-p} \label{eq:intro:chi_crit_exponent}}
for $p<p_c$ and constants $0<c<C$. We write $\chi(p) \sim (p_c-p)^{-1}$ as $p\nearrow p_c$ for the behavior of $\chi$ as in~\eqref{eq:intro:chi_crit_exponent}. There are several other critical exponents that are predicted to exist. For example, $\theta(p) \sim (p-p_c)^\beta$ as $p \searrow p_c$, and $\p_{p_c}(|\C(\orig)| \geq n) \sim n^{-1/\delta}$ as $n\to\infty$. 

Barsky and Aizenman~\cite{AizBar91} show that under the triangle condition,
	\eqq{ \delta =2 \qquad \text{and} \qquad \beta=1. \label{eq:intro:beta_delta}}
Their results are stated for a class of percolation models including site percolation. Hence, Theorem~\ref{thm:main_theorem_triangle_condition} implies~\eqref{eq:intro:beta_delta}. However, ``for simplicity of presentation'', the presentation of the proofs is restricted to bond percolation models.

Moreover, as shown by Nguyen~\cite{Ngu87}, Theorem~\ref{thm:main_theorem_triangle_condition} implies that $\Delta=2$, where $\Delta$ is the gap exponent.

\subsection{Discussion of literature and results}\label{sec:intro:literature}
Percolation theory is a fundamental part of contemporary probability theory and its foundations are generally attributed to a 1957 paper of Broadbent and Hammersley \cite{BroHam57}. 
Meanwhile, a number of textbooks appeared, and we refer to Grimmett~\cite{Gri99} for a comprehensive treatment of the subject, as well as Bollob{\'a}s and Riordan~\cite{BolRio06}, Werner~\cite{Wer09} and Beffara and Duminil-Copin~\cite{BefDum13} for extra emphasis on the exciting recent progress in two-dimensional percolation.

The investigation of percolation in high dimensions was started by the seminal 1990 paper of Hara and Slade~\cite{HarSla90}, who applied the lace expansion to prove the triangle condition for bond percolation in high dimension. 
A number of modifications and extensions of the lace expansion method for bond percolation have appeared in the meantime. The expansion itself is presented in Slade's Saint Flour notes \cite{Sla06}. A detailed account of the full lace expansion proof for bond percolation (including convergence of the expansion and related results) is given in a recent textbook by the first author and van der Hofstad~\cite{HeyHof17}. 

Despite the fantastic understanding of \emph{bond} percolation in high dimensions, \emph{site} percolation is not yet analyzed with this method, and the present paper aims to remedy this situation. 
Together with van der Hofstad and Last~\cite{HeyHofLasMat19}, we recently applied the lace expansion to the random connection model, which can be viewed as a continuum site percolation model. The aim of this paper is to give a rigorous exhibition of the lace expansion applied to one of the simplest site percolation lattice models. 
We have chosen to set up the proofs in a similar fashion to corresponding work for bond percolation, making it easier to oversee by readers who are familiar with that literature. Indeed, Sections \ref{sec:diag_bounds} and \ref{sec:bootstrap_analysis} follow rather closely the well-established method in \cite{BorChaHofSlaSpe05,HeyHof17,HeyHofSak08}, which are all based on Hara and Slade's foundational work \cite{HarSla90}. Interestingly, there is a difference in the expansion itself, which has numerous repercussions in the diagrammatic bounds and also in the different form of the infrared bounds. We now explain these differences in more detail. 

A key insight for the analysis of high-dimensional site percolation is the frequent occurrence of \emph{pivotal points}, which are crucial for the setup of the lace expansion. Suppose two sites $x$ and $y$ are connected, then an intermediate vertex $u$ is \emph{pivotal} for this connection if every occupied path from $x$ to $y$ passes through $u$. We then break up the original connection event in two parts: a connection between $x$ and $u$, and a second connection between $u$ and $y$. In high dimension, we expect the two connection events to behave rather independently, except for their joint dependence on the occupation status of $u$.
This little thought demonstrates that it is highly convenient to define connectivity events as in Section~\ref{sec:defs}, and thus treat the occupation of the vertex $u$ independent of the two new connection events. 
A more conventional choice of connectivity, where two points can only be connected if they are both occupied, is obtained a posteriori via
	\eqq{ \pp(\{x\longleftrightarrow y \} \cap \{ x,y\text{ occupied}\})=p^2\,\tau_p(y-x),\qquad x,y\in\Zd, x\neq y. \label{eq:intro:tlam_p_squared_alt_def}}
Our definition of $\taup$ avoids divisions by $p$, not only in the triangle, but throughout Sections~\ref{sec:lace_expansion} and~\ref{sec:diag_bounds}. 

Paying close attention to the right number of factors of $p$ is a guiding thread of the technical aspects of Sections~\ref{sec:diag_bounds} and~\ref{sec:bootstrap_analysis} that sets site percolation apart from bond percolation. Like it is done in bond percolation, the diagrammatic events by which we bound the lace-expansion coefficients depend on quite a few more points than the pivotal points. In contrast, every pair of points among these that may coincide hides a case distinction, and a coincidence case leads to a new diagram, typically with a smaller number of factors of $p$ (see, e.g.,~\eqref{eq:db:Pi_bound_by_F}). We handle this, for example, by encoding such coincidences in $\taupo$ and $\taupf$ (which are variations of $\taup$ respecting an extra factor of $p$ and/or interactions at $\orig$), see Definition~\ref{def:db:modified_two-point_functions}. In Section~\ref{sec:bootstrap_analysis}, the mismatching number of $p$'s and $\taup$'s in $\trip$ needs to be resolved.

The differing diagrams in Section~\ref{sec:diag_bounds} due to coincidences already appear in the lace-expansion coefficients of small order. This manifests itself in the answer to a classical question for high-dimensional percolation; namely, to devise an expansion of the critical threshold $p_c(d)$ when $d\to\infty$. It is known in the physics literature that
	\begin{equation}\label{eq:expansion-phys}
	p_c(d) = (2d)^{-1} + \frac 52 (2d)^{-2} + \frac{31}{4}(2d)^{-3} + \frac{75}{4} (2d)^{-4} + \frac{11977}{48}(2d)^{-5}+ \frac{209183}{96}(2d)^{-6}+\cdots.
	\end{equation}
The first four terms are due to Gaunt, Ruskin and Sykes~\cite{GauRusSyk76}, the latter two were found recently by Mertens and Moore~\cite{MerMoo18} by exploiting involved numerical methods. 

The lace expansion devised in this paper enables us to give a rigorous proof of the first terms of \eqref{eq:expansion-phys}. Indeed, we use the representation obtained in this paper to show that 
	\eqq{ p_c(d) = (2d)^{-1} + \frac 52 (2d)^{-2} + \frac{31}{4} (2d)^{-3} + \mathcal O\left( (2d)^{-4} \right) \quad\text{as $d \to \infty$}.  \label{eq:expansion}}
This is the content of a separate paper~\cite{HeyMat19b}. Deriving $p_c$ expansions from lace expansion coefficients has been earlier achieved for bond percolation by Hara and Slade~\cite{HarSla95} and van der Hofstad and Slade~\cite{HofSla06}. \col{Comparing~\eqref{eq:expansion} to their expansion for bond percolation confirms that already the second coefficient is different.

Proposition~\ref{thm:convergence_of_LE} proves the convergence of the lace expansion for $p<p_c$, yielding an identity for $\taup$ of the form
	\eqq{ \taup(x) = C(x) + p(C \ast \taup)(x), \label{eq:intro:OZE}}
where $C$ is the \emph{direct-connectedness function} and $C(\cdot)=2d D(\cdot) + \Pi_p(\cdot)$ (for a definition of $\Pi_p$, see Definition~\ref{def:le:lace_expansion_coefficients} and Proposition~\ref{thm:convergence_of_LE}). In fluid-state statistical mechanics,~\eqref{eq:intro:OZE} is known as the \emph{Ornstein-Zernike equation} (OZE), a classical equation that is typically associated to the total correlation function. 

We can juxtapose~\eqref{eq:intro:OZE} with the converging lace expansion for bond percolation. There, $\taup^{\text{bond}}(x)$ is the probability that $\orig$ and $x$ are connected by a path of occupied bonds, and we have
	\eqq{ \taup^{\text{bond}}(x) = C^{\text{bond}}(x) + 2dp (C^{\text{bond}} \ast\connf\ast\ \taup^{\text{bond}})(x), \label{eq:intro:OZE_bond}}
where $C^{\text{bond}}(x) = \mathds 1_{\{x = \orig\}} + \Pi_p^{\text{bond}}(x)$. Thus, there is an extra convolution with $2dD$. Only for the site percolation two-point function (as defined in this paper), the lace expansion coincides with the OZE.

We want to touch on how this relates to the infra-red bound~\eqref{eq:intro:infra-red_bound}. To this end,} define the random walk Green's function as $\greenslam(x) = \sum_{m \geq 0} \lambda^m D^{\ast m}(x)$ for $\lambda\in(0,1]$. Consequently,
	\[ \fgreenslam(k) = \frac{1}{1-\lambda \fconnf(k)}.\]
One of the key ideas behind the lace expansion for bond percolation is to show that the two-point function is close to $\greenslam$ in an appropriate sense (this includes an appropriate parametrization of $\lambda$). Solving the OZE in Fourier space for $\ftau$ already hints at the fact that in site percolation, $p\ftau$ should be close to $ \fgreenslam\widehat D$ and $p\taup$ should be close to $D\ast\greenslam$. As a technical remark, we note that $\fgreenslam$ is uniformly lower-bounded, whereas $\fgreenslam \fconnf$ is not, which poses some inconvenience later on.

The complete graph may be viewed as a mean-field model for percolation, in particular when we analyze clusters on high-dimensional tori, cf.\ \cite{HeyHof07}. Interestingly, the distinction between bond and site percolation exhibits itself rather drastically on the complete graph: for bond percolation, we obtain the usual Erd\H{o}s-R\'enyi random graph with its well-known phase transition (see, e.g.,~\cite{JanLucRuc00}), whereas for site percolation, we obtain again a complete graph with a binomial number of points.

Theorem \ref{thm:main_theorem_triangle_condition} proves the triangle condition in dimension $d> d_0$ for sufficiently large $d_0$. It is folklore in the physics literature that $d_0=6$ suffices (6 is the ``upper critical dimension'') but the perturbative nature of our argument does not allow us to derive that. Instead, we only get the result for \emph{some} $d_0\ge 6$. 
For bond percolation, already the original paper by Hara and Slade \cite{HarSla90} treated a second, spread-out version of bond percolation, and they proved that for this model, $d_0=6$ suffices (under suitable assumption on the spread-out nature). For ordinary bond percolation, it was announced that $d_0=19$ suffices for the triangle condition in \cite{HarSla94}, and the number 19 circulated for many years in the community. Finally, Fitzner and van der Hofstad~\cite{FitHof17} devised involved numerical methods to rigorously verify that an adaptation of the method is applicable for $d > d_0=10$. 
It is clear that an analogous result of Theorem \ref{thm:main_theorem_triangle_condition} would hold for ``spread-out site percolation'' in suitable form (see e.g. \cite[Section 5.2]{HeyHof17}).

\subsection{Outline of the paper} \label{sec:intro:outine}
The paper is organized as follows. The aim of Section~\ref{sec:lace_expansion} is to establish a lace-expansion identity for $\taup$, which is formulated in Proposition~\ref{thm:lace_expansion}. To this end, we use Section~\ref{sec:le:tools} to state some known results that we are going to make use of in Section~\ref{sec:lace_expansion} as well as in later sections. We then introduce a lot of the language and quantities needed to state Proposition~\ref{thm:lace_expansion} in Section~\ref{sec:le:preparatory_statements}, followed by the actual derivation of the identity in Section~\ref{sec:le:derivation_of_the_LE}.

Section~\ref{sec:diag_bounds} bounds the lace-expansion coefficients derived in Section~\ref{sec:le:derivation_of_the_LE} in terms of simpler diagrams, which are large sums over products of two-point (and related) functions. Section~\ref{sec:bootstrap_analysis} finishes the argument via the so-called bootstrap argument. First, a \emph{bootstrap function} $f$ is introduced in Section~\ref{sec:boot:intro}. Among other things, it measures how close $\ftau$ is to $\fgreenslam$ (in a fractional sense). Section~\ref{sec:boot:consequences} shows convergence of the lace expansion for fixed $p<p_c$. Moreover, assuming that $f$ is bounded on $[0,p_c)$, it is shown that this convergence is uniform in $p$ (see first and second part of Proposition~\ref{thm:convergence_of_LE}). Lastly, Section~\ref{sec:boot:bootstrap_argument} actually proves said boundedness of $f$.

\section{The expansion}\label{sec:lace_expansion}

\subsection{The standard tools} \label{sec:le:tools}
We require two standard tools of percolation theory, namely Russo's formula and the BK inequality, both for increasing events. Recall that $A$ is called \emph{increasing} if $\omega \in A$ and $\omega \subseteq \omega'$ implies $\omega' \in A$. Given $\omega$ and an increasing event $A$, we introduce
	\[ \piv{A} = \{ y \in \Zd: \omega \cup\{y\} \in A, (\omega\setminus\{y\}) \notin A\}.\]
If $A$ is an increasing event determined by sites in $\Lambda \subset \Zd$ with $|\Lambda|<\infty$, then Russo's formula~\cite{Rus81}, proved independently by Margulis~\cite{Mar74}, tells us that
	\eqq{ \frac{\dd}{\dd p} \pp(A) = \E[|\piv{A}|] = \sum_{y \in \Lambda} \pp ( y \in \piv{A}). \label{eq:tools:russo}}
To state the BK inequality, let $\Lambda\subset \Zd$ be finite and, given $\omega\in\Omega$, let
	\[ [\omega]_\Lambda = \{ \omega' \in \Omega: \omega'(x) = \omega(x) \text{ for all } x \in\Lambda\} \]
be the cylinder event of the restriction of $\omega$ to $\Lambda$. For two events $A,B$, we can define the \emph{disjoint occurrence} as
	\[ A \circ B = \{ \omega: \exists K,L \subseteq \Zd: K \cap L = \varnothing, [\omega]_K \subseteq A, [\omega]_L \subseteq B\}.\]
The BK inequality, proved by van den Berg and Kesten~\cite{BerKes85} for increasing events, states that, given two increasing events $A$ and $B$,
	\eqq{ \pp(A \circ B) \leq \pp(A) \pp(B). \label{eq:tools:bk}}
The following proposition about simple random walk will be of importance later:
\begin{prop}[Random walk triangle,~\cite{HeyHof17}, Proposition 5.5] \label{thm:random_walk_triangle}
Let $m \in\N_0, n\geq 0$ and $\lambda\in[0,1]$. Then there exists a constant $c_{2m,n}^{\text{(RW)}}$ independent of $d$ such that, for $d>2n$,
	\[ \int_{\fspace} \frac{ \fconnf(k)^{2m}}{[1-\lambda\fconnf(k)]^n} \frac{\dd k}{(2\pi)^d} \leq c_{2m,n}^{\text{(RW)}} d^{-m}.  \]
\end{prop}

In~\cite{HeyHof17}, $d>4n$ is required; however, more careful analysis shows that $d>2n$ suffices (see~\cite[(2.19)]{BorChaHofSlaSpe05}). We will also need the following related result:
\begin{prop}[Related random walk bounds, \cite{HeyHof17}, Exercise 5.4] \label{thm:random_walk_triangle_related}
Let $m \in\{0,1\}$, $\lambda\in[0,1]$, and $r,n \geq 0$ such that $d>2(n+r)$. Then, uniformly in $k\in\fspace$,
	\al{ \int_{\fspace} \fconnf(l)^{2m} \fgreenslam(l)^n \tfrac 12 \big[ \fgreenslam(l+k) + \fgreenslam(l-k) \big]^{r} \frac{\dd l}{(2\pi)^d} & \leq c^{\text{(RW)}}_{2m,n+r} d^{-m}, \\
		\int_{\fspace} \fconnf(l)^{2m} \fgreenslam(l)^{n-1} \big[ \fgreenslam(l+k) \fgreenslam(l-k) \big]^{r/2} \frac{\dd l}{(2\pi)^d} & \leq c^{\text{(RW)}}_{2m,n-1+r/2} d^{-m}, }
where the constants $c^{\text{(RW)}}_{\cdot,\cdot}$ are from Proposition~\ref{thm:random_walk_triangle}.
\end{prop}

The following differential inequality is an application of Russo's formula and the BK inequality. It applies them to events which are not determined by a finite set of sites. We refer to the literature~\cite[Lemma 4.4]{HeyHof17} for arguments justifying this and for a more detailed proof. Observation~\ref{obs:tools:diff_inequality} will be of use in Section~\ref{sec:bootstrap_analysis}.
\begin{observation} \label{obs:tools:diff_inequality}
Let $p < p_c$. Then
	\[ \frac{\dd}{\dd p} \ftau(0) \leq \ftau(0)^2, \qquad \frac{\dd}{\dd p} \chi(p) \leq \chi(p) \ftau(0).\]
\end{observation}
As a proof sketch, note that
	\al{ \frac{\dd}{\dd p} \ftau(0) &= \sum_{x \in\Zd} \frac{\dd}{\dd p} \taup(x) = \sum_{x \in\Zd} \sum_{y\in\Zd} \pp\big(y \in \piv{\orig \longleftrightarrow x} \big)
			\leq \sum_{x \in\Zd} \sum_{y\in\Zd} \pp\big( \{\orig \longleftrightarrow y\}\circ\{y \longleftrightarrow x\} \big) \\
			& \leq \sum_{x \in\Zd} \sum_{y\in\Zd} \taup(y) \taup(x-y) = \ftau(0)^2.  }
The inequality for $\chi(p)$ follows from the identity $\chi(p) = 1+p \ftau(0)$.

\subsection{Definitions and preparatory statements}\label{sec:le:preparatory_statements}

We need the following definitions:
\begin{definition}[Elementary definitions] \label{def:le:elementary_definitions} Let $x,u\in\Zd$ and $A \subseteq \Zd$.
\begin{enumerate}
\item We set $\omega^x := \omega \cup \{x\}$ and $\omega^{u,x} := \omega \cup \{u,x\}$.
\item We define $\jeq(x) := \mathds 1_{\{|x|=1\}} = 2d D(x)$.
\item Let $\{\conn{u}{x}{A}\}$ be the event that $\{u \longleftrightarrow x\}$, and there is a path from $u$ to $x$, all of whose inner vertices are elements of $\omega\cap A$. Moreover, write $\{\offconn{u}{x}{A}\} :=\{\conn{u}{x}{\Zd \setminus A}\}$.
\item We define $\{u \Longleftrightarrow x\} := \{u \longleftrightarrow x\} \circ \{u \longleftrightarrow x\} $ and say that $u$ and $x$ are \emph{doubly connected}.
\item We define the modified cluster of $x$ with a designated vertex $u$ as
		\[\widetilde\C^{u}(x) := \{x\} \cup \{y \in \omega \setminus\{u\} : x \longleftrightarrow y \text{ in } \Zd \setminus\{u\}  \} . \]
\item For a set $A \subset \Zd$, define $\langle A \rangle := A \cup \{y \in \Zd: \exists x \in A \text{ s.t.~} |x-y| = 1\}$ as the set $A$ itself plus its external boundary.
\end{enumerate}
\end{definition}
Definition~\ref{def:le:elementary_definitions}.1 allows us to speak of events like $\{\conn{a}{b}{\omega^x}\}$ for $a,b\in\Zd$, which is the event that $a$ is connected to $b$ in the configuration where $x$ is fixed to be occupied. We remark that $\{\conn{x}{y}{\Zd}\} = \{x \longleftrightarrow y\} = \{\conn{x}{y}{\omega}\}$ and that $\{u \Longleftrightarrow x\} = \Omega$ for $|u-x| = 1$. Similarly, $\{u \Longleftrightarrow x\} = \varnothing$ for $u=x$.
The following, more specific definitions are important for the expansion:
\begin{definition}[Extended connection probabilities and events] \label{def:le:extended_connection_stuff} Let $v,u,x \in\Zd$ and $A \subseteq \Zd$.
\begin{enumerate}
\item Define
	\[ \{u \throughconn{A} x \} := \{u \longleftrightarrow x\} \cap \Big( \{\noffconn{u}{x}{\thinn{A}}\} \cup  \{x \in \thinn{A}  \}  \Big). \]
In words, this is the event that $u$ is connected to $x$, but either any path from $u$ to $x$ has an interior vertex in $\thinn{A}$, or $x$ itself lies in $\thinn{A}$.
\item Define
	\[ \taup^{A}(u,x) := \mathds 1_{\{x \notin \thinn{A}\}}\pp(\offconn{u}{x}{\thinn{A}}). \]
\item We introduce $\piv{u,x} := \piv{u \longleftrightarrow x}$ as the set of pivotal points for $\{u \longleftrightarrow x\}$. That is, $v \in\piv{u,x}$ if the event $\{\conn{u}{x}{\omega^v}\}$ holds but $\{\conn{u}{x}{\omega\setminus\{v\}}\}$ does not.
\item Define the events
	\al{ E'(v,u;A) &:= \{v \throughconn{A} u\} \cap \{ \nexists u' \in \piv{v,u}: v \throughconn{A} u'\}, \\
			E(v,u,x;A) &:= E'(v,u;A) \cap \{u \in \omega \cap \piv{v,x} \}.  }
\end{enumerate}
\end{definition}
First, we remark that $\{u \throughconn{\Zd} x\} = \{ u \longleftrightarrow x\}$. Secondly, note that we have the relation
	\eqq{ \taup(x-u)= \taup^{A}(u,x)+ \pp(u \throughconn{A} x).   \label{eq:le:taup_incl_excl_identity}}
We next state a partitioning lemma (whose proof is left to the reader; see~\cite[Lemma 3.5]{HeyHofLasMat19}) relating the events $E$ and $E'$ to the connection event $ \{u \throughconn{A} x \}$:
\begin{lemma} [Partitioning connection events] \label{lem:partitioning_lemma}
Let $v,x\in\Zd$ and $A \subseteq \Zd$. Then
	\[ \{v \throughconn{A} x \} = E'(v,x;A) \cup \bigcup_{u\in \Zd} E(v, u, x; A),  \]
and the appearing unions are disjoint.
\end{lemma}
The next lemma, titled the Cutting-point lemma, is at the heart of the expansion:
\begin{lemma}[Cutting-point lemma] \label{lem:cutting_point_lemma}
Let $v,u,x\in\Zd$ and $A \subseteq \Zd$. Then
	\[ \pp(E(v,u,x;A)) = p \E_p \left[\mathds 1_{E'(v,u;A)} \taup^{\widetilde\C^u(v)}(u,x)  \right].\]
\end{lemma}
\begin{proof}
The proof is a special case of the general setting of \cite{HeyHofLasMat19}. Since it is essential, we present it here. We abbreviate $\widetilde\C = \widetilde\C^u(v)$ and observe that
	\al{ \{u \in \piv{v,x} \} &= \{v \longleftrightarrow u\} \cap \{\offconn{u}{x}{\widetilde\C} \} \cap \{x \notin \thinn{\widetilde\C} \} \\
			&= \{v \longleftrightarrow u\} \cap \{\offconn{u}{x}{\thinn{\widetilde\C}} \} \cap \{ x \notin \thinn{\widetilde\C} \} .}
In the above, we can replace $\widetilde\C$ by $\thinn{\widetilde\C}$ in the middle event, as, by definition, we know that, apart from $u$, any site in $\thinn{\widetilde \C} \setminus \widetilde \C$ must be vacant. Now, since $E'(v,u;A) \subseteq \{ v \longleftrightarrow u\}$, we get
	\[ E(v,u,x;A) = E'(v,u;A) \cap \{\offconn{u}{x}{\thinn{\widetilde\C}} \} \cap \{x \notin \thinn{\widetilde\C} \}\cap \{ u\in\omega\}.\]
Taking probabilities, conditioning on $\widetilde \C$, and observing that the status of $u$ is independent of all other events, we see
	\[ \pp(E(v,u,x;A)) = p \E_p \left[ \mathds 1_{E'(v,u;A)} \mathds 1_{\{x \notin \thinn{\widetilde\C}\}} \E_p \left[ \mathds 1_{\{\offconn{u}{x}{\thinn{\widetilde\C}} \}} | \widetilde \C \right] \right],\]
making use of the fact that the first two events are measurable w.r.t.~$\widetilde\C$. The proof is complete with the observation that under $\E_p$, almost surely,
	\[  \mathds 1_{\{x \notin \thinn{\widetilde\C} \}} \E_p \left[ \mathds 1_{\{\offconn{u}{x}{\thinn{\widetilde\C}} \}} | \widetilde \C \right] =  \taup^{\widetilde\C}(u,x). \qedhere \]
\end{proof}

\subsection{Derivation of the expansion} \label{sec:le:derivation_of_the_LE}
We introduce a sequence $(\omega_i)_{i\in\N_0}$ of independent site percolation configurations. For an event $E$ taking place on $\omega_i$, we highlight this by writing $E_i$. We also stress the dependence of random variables on the particular configuration they depend on. For example, we write $\C(u; \omega_i)$ to denote the cluster of $u$ in configuration $i$.
\begin{definition}[Lace-expansion coefficients] \label{def:le:lace_expansion_coefficients}
Let $m\in\N, n\in\N_0$ and $x\in\Zd$. We define
	\al{ \Pi_p^{(0)}(x) &:= \pp(\orig \Longleftrightarrow x) - \jeq(x), \\
			\Pi_p^{(m)}(x) &:= p^m \sum_{u_0, \ldots, u_{m-1}} \pp \Big( \{\orig \Longleftrightarrow u_0\}_0 \cap \bigcap_{i=1}^{m} E'(u_{i-1},u_i; \mathscr C_{i-1})_i \Big), }
where we recall that $\jeq(x) = \mathds 1_{\{|x|=1\}}$ and moreover $u_{-1}=\orig, u_m=x$, and $\C_{i} = \widetilde\C^{u_i}(u_{i-1};\omega_i)$. Let
	\al{R_{p,n}(x) &:= (-p)^{n+1} \sum_{u_0, \ldots, u_n} \pp \Big( \{\orig \Longleftrightarrow u_0\}_0 \cap \bigcap_{i=1}^{n} E'(u_{i-1},u_i; \mathscr C_{i-1})_i
							\cap \{u_n \throughconn{\C_n} x\}_{n+1} \Big).}
Finally, set
	\[ \Pi_{p,n}(x) = \sum_{m=0}^{n} (-1)^m \Pi_p^{(m)}(x).\]
\end{definition}

It should be noted that the events $E'(u_{i-1},u_i; \mathscr C_{i-1})_i$ appearing in Definition~\ref{def:le:lace_expansion_coefficients} take place on configuration $i$ only if $\C_{i-1}$ is taken to be a fixed set---otherwise, they are events determined by configurations $i-1$ and $i$.

\begin{prop}[The lace expansion] \label{thm:lace_expansion}
Let $p< p_c$, $x\in\Zd$, and $n\in\N_0$. Then
	\[ \taup(x) = \jeq(x) + \Pi_{p,n}(x) + p \big((\jeq+\Pi_{p,n})\ast \taup) (x) + R_{p,n}(x).  \]
\end{prop}
\begin{proof}
We have
	\[\taup(x) = \jeq(x) + \Pi_p^{(0)}(x) + \pp(\orig \longleftrightarrow x, \orig \centernot\Longleftrightarrow x).\]
We can partition the last summand via the first pivotal point. Pointing out that $\{\orig \Longleftrightarrow u\} = E'(0,u;\Zd)$, we obtain
	\al{ \pp(\orig \longleftrightarrow x, \orig \centernot\Longleftrightarrow x) 
			&= \sum_{u_0 \in\Zd} \pp(\orig \Longleftrightarrow u_0, u_0 \in \omega, u_0\in\piv{\orig, x})  = \sum_{u_0} \pp(E(\orig,u_0,x;\Zd)) \\
			&= p \sum_{u_0} \E_p\left[ \mathds 1_{\{\orig \Longleftrightarrow u_0\}_0} \cdot \taup^{\C_0} (u_0,x) \right] }
via the Cutting-point lemma~\ref{lem:cutting_point_lemma}. Using~\eqref{eq:le:taup_incl_excl_identity} for $A=\C_0$, we have
	\eqq{\taup(x) = \jeq(x) + \Pi_p^{(0)}(x) + p \sum_{u_0} \left(\jeq(u_0) + \Pi_p^{(0)}(u_0)\right) \taup(x-u_0)
												 - p \sum_{u_0} \E_p\left[ \mathds 1_{\{\orig \Longleftrightarrow u_0\}_0} \cdot
												\pp(u_0 \throughconn{\C_0} x) \right] . \label{eq:lace_expansion_step_0}}
This proves the expansion identity for $n=0$. Next, Lemma~\ref{lem:partitioning_lemma} and Lemma~\ref{lem:cutting_point_lemma} yield
	\al{\pp(u_0\throughconn{A} x) &= \pp(E'(u_0,x;A)) + \sum_{u_1\in \Zd} \pp(E(u_0,u_1,x;A)) \\
		&=  \pp(E'(u_0,x;A)) + p \sum_{u_1\in \Zd} \E_p \left[ \mathds 1_{E'(u_0,u_1;A)} \cdot \taup^{\widetilde\C^{u_1}(u_0)}(u_1,x) \right]. }
Plugging this into~\eqref{eq:lace_expansion_step_0}, we use~\eqref{eq:le:taup_incl_excl_identity} for $A=\widetilde \C^{u_1}(u)$ to extract $\Pi_p^{(1)}$ and get
	\al{ \taup(x) & = \jeq(x) + \Pi_p^{(0)}(x) - \Pi_p^{(1)}(x) + p  \big((\jeq + \Pi_p^{(0)}) \ast\taup\big)(x) \\
			& \quad + p^2 \sum_{u_1} \taup(x-u_1) \sum_{u_0} \pp \Big( \{\orig \Longleftrightarrow u_0\}_0 \cap E'(u_0,u_1;\C_0)_1 \Big) + R_{p,1}(x) \\
			&= \jeq(x) + \Pi_p^{(0)}(x) - \Pi_p^{(1)}(x) + p  \big((\jeq + \Pi_p^{(0)}-\Pi_p^{(1)}) \ast\taup\big)(x) + R_{p,1}(x) .}
Note that all appearing sums are bounded by $\sum_{y} \taup(y)$. This sum is finite for $p<p_c$, justifying the above changes in order of summation. The expansion for general $n$ is an induction on $n$ where the step is analogous to the step $n=1$ (but heavier on notation).
\end{proof}

\section{Diagrammatic bounds} \label{sec:diag_bounds}
\subsection{Setup, bounds for $n=0$} \label{sec:db:n0}
We use this section to state Lemma~\ref{lem:cosinesplitlemma} and state bounds on $\Pi_p^{(0)}$, which are rather simple to prove. The more involved bounds on $\Pi_p^{(n)}$ for $n\geq 1$ are dealt with in Section~\ref{sec:db:no_displacement}. Note that if $f(-x) = f(x)$, then $\widehat f(k) = \sum_{x \in \Zd} \cos (k \cdot x) f(x)$. We furthermore have the following tool at our disposal:
\begin{lemma}[Split of cosines, \cite{FitHof16}, Lemma 2.13] \label{lem:cosinesplitlemma}
Let $t \in \R$ and $t_i \in \R$ for $i=1, \ldots, n$ such that $t = \sum_{i=1}^{n} t_i$. Then
	\[1-\cos (t) \leq n \sum_{i=1}^{n} [1- \cos(t_i)]. \]
\end{lemma}

We begin by treating the coefficient for $n=0$, giving a glimpse into the nature of the bounds to follow in Sections~\ref{sec:db:no_displacement} and~\ref{sec:db:displacement}. To this end, we define the two displacement quantities
	\[\jek(x) := [1-\cos(k\cdot x)] \jeq(x) \quad \text{and} \quad \taupk(x) = [1-\cos(k\cdot x)] \taup(x). \]
\begin{prop}[Bounds for $n=0$] \label{thm:db:bounds_for_n0}
For $k\in \fspace$,
	\al{|\widehat\Pi_p^{(0)}(k)| & \leq p^2 (\jeq^{\ast 2} \ast \taup^{\ast 2})(\orig), \\
			|\widehat\Pi_p^{(0)}(0)-\widehat\Pi_p^{(0)}(k)| & \leq 2p^2 \Big((\jek\ast\jeq\ast\taup^{\ast 2})(\orig) + (\jeq^{\ast 2}\ast\taupk\ast\taup)(\orig)  \Big). }
\end{prop}
\begin{proof}
Note that $|x| \leq 1$ implies $\Pi_p^{(0)}(x) =0$ by definition. For $|x| \geq 2$, we have
	\[\Pi_p^{(0)}(x)  \leq \E \bigg[\sum_{y \neq z \in \omega} \mathds 1_{\{|y|=|z|=1\}}\mathds 1_{\{y \longleftrightarrow x\}\circ\{z \longleftrightarrow x\}}  \bigg]
			\leq p^2 \bigg( \sum_{y\in\Zd} \jeq(y)\taup(x-y) \bigg)^2 = p^2 (\jeq\ast\taup)(x)^2.\]
Summation over $x$ gives the first bound. The last bound is obtained by applying Lemma~\ref{lem:cosinesplitlemma} to the bounds derived for $\Pi_p^{(0)}(x)$:
	\al{ |\widehat\Pi_p^{(0)}(0)&-\widehat\Pi_p^{(0)}(k)| = \sum_{x} [1-\cos(k\cdot x)] \Pi_p^{(0)}(x) \\
			& \leq 2p^2  \sum_x (J\ast\taup)(x) \sum_{y} \Big( [1-\cos(k\cdot y)] J(y) \taup(x-y) + J(y) [1-\cos(k\cdot(x-y))] \taup(x-y) \Big).}
Resolving the sums gives the claimed convolution. \end{proof}

\subsection{Bounds in terms of diagrams} \label{sec:db:no_displacement}

The main result of this section is Proposition~\ref{thm:db:main_thm}, providing bounds on the lace-expansion coefficients in terms of so-called diagrams, which are sums over products of two-point (and related) functions. To state it, we introduce some functions related to $\taup$ as well as several ``modified triangles'' closely related to $\trip$. We denote by $\delta_{x,y} = \mathds 1_{\{x=y\}}$ the Kronecker delta.

\begin{definition}[Modified two-point functions] \label{def:db:modified_two-point_functions}
Let $x\in\Zd$ and define
	\[ \taupo(x) := \delta_{\orig,x} + \taup(x), \qquad \taupf(x) = \delta_{\orig,x} + p\taup(x), \qquad \ttaup(x) = \taup(x)-\jeq(x).\]
\end{definition}
\begin{definition}[Modified triangles] \label{def:db:modified:triangles}
We let $\tripo(x) = p^2(\taupo\ast\taup\ast\taup)(x), \tripf(x) = p(\taupf\ast\taup\ast\taup)(x), \tripof(x) = p(\taupf\ast\taupo\ast\taup)(x)$, and $\tripoff(x) = (\taupf\ast\taupf\ast\taupo)(x)$. We also set
	\[\tripo = \sup_{x \in \Zd} \tripo(x), \quad \tripf = \sup_{\orig \neq x \in\Zd} \tripf(x), \quad \tripof = \sup_{\orig \neq x \in\Zd} \tripof(x), \quad \tripoff = \sup_{x\in\Zd} \tripoff(x),  \]
and $T_p := (1+\trip) \tripof + \trip \tripoff$.
\end{definition}

Definitions~\ref{def:db:modified_two-point_functions} and~\ref{def:db:modified:triangles} allow us to properly keep track of factors of $p$, which turns out to be important throughout Section~\ref{sec:diag_bounds}.

\begin{prop}[Triangle bounds on the lace-expansion coefficients] \label{thm:db:main_thm}
For $n\geq 0$, 
	\[ p\sum_{x \in \Zd} \Pi_p^{(n)}(x) \leq \tripf(\orig) \big( T_p  \big)^n.\]
\end{prop}

The proof of Proposition~\ref{thm:db:main_thm} relies on two intermediate steps, successively giving bounds on $\sum \Pi_p^{(n)}$. These two steps are captured in Lemmas~\ref{lem:db:Pi_bounds_by_F} and~\ref{lem:db:pi_psi_bounds}, respectively. We first state the former lemma.

Recall that $\Pi_p^{(n)}$ is defined on independent percolation configurations $\omega_0, \ldots, \omega_n$. A crucial step in proving Proposition~\ref{thm:db:main_thm} is to group events taking place on the percolation configuration $i$, and then to use the independence of the different configurations. To this end, note that event $E'(u_{i-1}, u_i; \C_{i-1})_i$ takes place on configuration $i$ only if $\C_{i-1}$ is considered to be a fixed set. Otherwise, it is a product event made up of the connection events of configuration $i$ as well as a connection event in configuration $i-1$, preventing a direct use of the independence of the $\omega_i$. Resolving this issue is one of the goals of Lemma~\ref{lem:db:Pi_bounds_by_F}; another is to give bounds in terms of the simpler events (amenable to application of the BK inequality) introduced below in Definition~\ref{def:Bounding_Events}:

\begin{figure}
	\centering
  \includegraphics{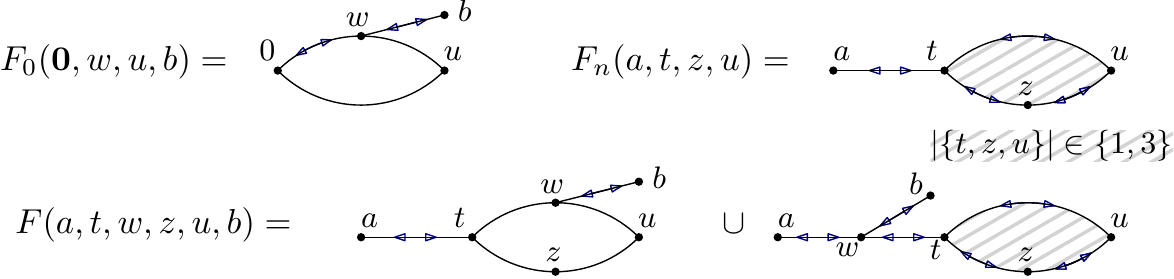}
	\caption{The $F$ events represented graphically. For lines with double arrows, we may have coincidence of the endpoints, for lines without double arrows, we do not. The area with grey tiles indicates that its three boundary points are either all distinct or collapsed into a single point. These diagrams also serve as a pictorial representation of the function $\phi_0, \phi, \phi_n$. There, lines with double arrows represent factors of $\taupo$ and lines without double arrows represent factors of $\taup$.}
	\label{fig:F_events}
\end{figure}
\begin{definition}[Bounding events] \label{def:Bounding_Events}
Let $x,y\in\Zd$. We define
	\[ \{x \oconn y \} := \{x \longleftrightarrow y\} \cup \{x=y\}. \]
Let now $i \in \{1, \ldots, n\}$ and let $a,b,t,w,z,u \in \Zd$. We define
	\al{F_0(a,w, u,b) &= \Big( \{w=a, |u-a| =1\} \cap \{ a \oconn b\}\Big) \\
		& \quad \cup \Big( \{|u-a| >1\} \cap \big(\{ a \oconn w\} \circ \{ a \longleftrightarrow u\} \circ\{ w \longleftrightarrow u\} \circ\{ w \oconn b\}\big)\Big), \\
			F_n(a, t, z, x) &= \{|\{t, z, x\}| \neq 2\} \cap \{ a \oconn t\} \\
				& \quad \circ \{ t \oconn x\} \circ\{ t \oconn z\} \circ\{ z \oconn x\}, \\
			F^{(1)} (a,t,w,z,u,b) &= \{|\{w, t, z, u\}| =4\} \cap \Big( \{ a \oconn t\} \circ\ \{ t \longleftrightarrow w\} \circ\ \{ t \longleftrightarrow z\} \\
				&\quad \circ\ \{ w \longleftrightarrow u\} \circ\ \{ z \longleftrightarrow u\} \circ\ \{ w \oconn b\}\Big), \\
			F^{(2)} (a,t,w,z,u,b) &= \big\{w \notin \{z, u\}, |\{t, z, u\}|  \neq 2 \big\} \cap \Big( \{ a \oconn w\} \circ\ \{ w \oconn b\} \\
				& \quad \circ \{w \oconn t \} \circ \{ t \oconn u\} \circ \{ t \oconn z\} \circ \{ z \oconn u\} \Big). }
\end{definition}
The coincidence requirements in $F^{(2)}$ mean that among the points $t_i, w_i, z_i, u_i$, the point $w_i$ may coincide \emph{only} with $t_i$; and additionally, the triple $\{t_i,z_i,u_i\}$ are either all distinct, or collapsed into a single point. The above events are depicted in Figure~\ref{fig:F_events}.

For intervals $[a,b]$, we use the notation $\vec x_{[a,b]} = (x_a, x_{a+1}, \ldots, x_b)$. This is not to be confused with the notation $\vec v_i$ from Definition~\ref{def:Bounding_Events}. We use the notation $(\Zd)^{(m,1)}$ to denote the set of vectors $\{\vec y_{[1,m]} \in (\Zd)^m: y_i \neq y_{i+1} \ \forall 1 \leq i < m\}$. 

\begin{lemma}[Coefficient bounds in terms of $F$ events] \label{lem:db:Pi_bounds_by_F}
For $n \geq 1$ and $(u_0, \ldots, u_{n-1}, x) \in (\Zd)^{(n+1,1)}$, 
	\al{ \{\orig \Longleftrightarrow u_0\}_0 & \cap \bigcap_{i=1}^{n} E'(u_{i-1}, u_i; \C_{i-1})_i \\
		 \subseteq \bigcup_{\substack{\vec z_{[1,n]}: z_i \in \omega_i^{u_i}, \\ w_0 \in \omega_0^\orig, t_n \in \omega_n^{u_{n-1}, x}}} & \hspace{-1cm} F_0(\orig,w_0, u_0, z_1)_0  \cap 
		 					\bigg( \bigcap_{i=1}^{n-1} \Big( \bigcup_{\substack{ t_i \in \omega_i^{u_{i-1}},\\ w_i \in \omega_i}} F^{(1)}(\vec v_i)_i \cup 
		 						\bigcup_{\substack{w_i \in \omega_i^{u_{i-1}}, \\ t_i\in \omega_i^{u_i, w_i} }} F^{(2)}(\vec v_i)_i \Big) \bigg) \\
		 & \qquad \qquad \cap F_n(u_{n-1}, t_n, z_n, x)_n, }
where $\vec v_i = (u_{i-1}, t_i,w_i,z_i,u_i, z_{i+1})$ and $u_n = x$.
\end{lemma}
The proof is analogous to the one in~\cite[Lemma~4.12]{HeyHofLasMat19} and we do not perform it here. The second important lemma is Lemma~\ref{lem:db:pi_psi_bounds}, and its bounds are phrased in terms of the following functions:

\begin{definition}[The $\psi$ and $\phi$ functions] \label{def:db:psi_phi_fcts}
Let $n \geq 1$ and $a_1,a_2,b,w,t,u,z \in \Zd$. We define 

\al{ \psi_0(b,w,u) & := \delta_{b,w} p \jeq (u-b) + p \taupf(w-b) \ttaup(u-b) \taup(w-u), \\
		 \tilde\psi_0(b,w,u) &:= p \taupf(w-b) \taup(u-b) \taup(w-u), \\
		\psi_n(a_1,a_2,t,z,x) &:=  \mathds 1_{\{|\{t,z,x\}| \neq 2\}} \taupo(z-a_1) \taupf(t-a_2)\taupf(z-t)\taupf(x-t)\taupo(x-z). }
Moreover, we define
	\al{ \psi^{(1)}(a_1,a_2,t,w,z,u) &:= p^3 \mathds 1_{\{|\{t,w,z,u\}| =4\}} \taupo(z-a_1) \taupf(t-a_2) \taup(w-t) \taup(z-t)\taup(u-w)\taup(u-z),  \\
		 \psi^{(2)}(a_1,a_2,t,w,z,u) &:= \mathds 1_{\{w \notin \{z,u\}, |\{t,z,u\}|  \neq 2 \}} \taupo(z-a_1) \taupf(w-a_2) \taupf(t-w) \taupf(z-t)\taupf(u-t) \taupo(u-z),}
and $\psi := \psi^{(1)} + \psi^{(2)}$. Furthermore, for $j \in\{1,2\}$, let 
	\al{ \phi_0(b,w,u,z) & := \delta_{b,w} p \jeq (u-b)  \taupo(z-w) + p \taupf(w-b) \ttaup(u-b) \taup(w-u) \taupo(z-w), \\
		 \phi_n(a_2,t,z,x) & := \mathds 1_{\{|\{t,z,x\}| \neq 2\}} \taupf(t-a_2)\taupf(z-t)\taupf(x-t)\taupo(x-z), \\
		 \phi^{(j)}(a_2,t,w,z,u,b) &:= \frac{\taupo(b-w)}{\taupo(z)} \psi^{(j)}(\orig,a_2,t,w,z,u),  }
and $\tilde\phi_0(b,w,u,z) := \tilde\psi_0(b,w,u) \taupo(z-w)$ as well as $\phi:= \phi^{(1)} + \phi^{(2)}$.
\end{definition}

We remark that $\psi_0 \leq \tilde\psi_0$ as well as $\phi_0 \leq \tilde\phi_0$, and we are going to use this fact later on. In the definition of $\phi^{(j)}$, the factor $\taupo(z)$ cancels out. In that sense, $\phi^{(j)}$ is obtained from $\psi^{(j)}$ by ``replacing'' the factor $\taupo(z-a_1)$ by the factor $\taupo(b-w)$, and the two functions are closely related.

We first obtain a bound on $\Pi_p^{(n)}$ in terms of the $F$ events (this is Lemma~\ref{lem:db:Pi_bounds_by_F}). Bounding those with the BK inequality, we will naturally observe the $\phi$ functions (Lemma~\ref{lem:db:pi_psi_bounds}). To decompose them further, we would like to apply induction; for this purpose, the $\psi$ functions are much better-suited. By introducing both the $\phi$ and $\psi$ functions, we increase the readability throughout this section (and later ones).

\begin{definition}[The $\Psi$ function] \label{def:db:Psi_fct}
Let $w_n, u_n \in \Zd$ and define
\[ \Psi^{(n)}(w_n,u_n) := \sum_{\substack{\vec t_{[1,n]}, \vec w_{[0,n-1]}, \vec z_{[1,n]}, \vec u_{[0,n-1]}: \\ u_{n-1} \neq u_n}} 
				\psi_0(\orig,w_0,u_0) \prod_{i=1}^{n} \psi(w_{i-1},u_{i-1},t_i,w_i,z_i,u_i),\]
where $\vec t_{[1,n]}, \vec z_{[1,n]}, \vec w_{[0,n-1]} \in (\Zd)^n$ and $\vec u_{[0,n-1]} \in (\Zd)^{(n,1)}$.
\end{definition}

We remark that $\Psi^{(n)}(x,x)=0$, resulting from the fact that the $\phi$ functions in Definition~\ref{def:db:psi_phi_fcts} output $0$ for $w=u$. We are going to make use of this later.

\begin{lemma}[Bound in terms of $\psi$ functions] \label{lem:db:pi_psi_bounds}
For $n\geq 0$,
\eqq{ p \sum_{x\in\Zd} \Pi_p^{(n)}(x) \leq \sum_{w,u,t,z,x \in \Zd} \Psi^{(n-1)}(w,u) \psi_n(w,u,t,z,x) \leq \sum_{w,u \in \Zd} \Psi^{(n)} (w,u). \label{eq:db:Pi_bound_by_Psi} }
\end{lemma}
\begin{proof}
\col{Note first that according to Definition~\ref{def:le:lace_expansion_coefficients}, the function $\Pi_p^{(n)}$ contains the event $E'(u_{i-1}, u_i; \C_{i-1})$, which is itself contained in $\{u_{i-1} \longleftrightarrow u_i\}$. This is why we may assume that the points $u_0, \ldots, u_n$ in the definition of $\Pi_p^{(n)}$ satisfy $u_{i-1} \neq u_i$.} Together with Lemma~\ref{lem:db:Pi_bounds_by_F}, this yields a bound on $\Pi_p^{(n)}$ of the form 
	\algn{p\Pi_p^{(n)}(u_n) &\leq p^{n+1} \sum_{\vec u} \pp \bigg( \bigcup_{w_0, t_n, \vec z} F_0(\orig,w_0, u_0, z_1)_0  \cap 
									\bigcap_{i=1}^{n-1} \Big( \bigcup_{t_i, w_i} F^{(1)}(\vec v_i)_i \cup \bigcup_{t_i, w_i} F^{(2)}(\vec v_i)_i\Big) \notag\\
			& \qquad \qquad \qquad \qquad \cap F_n(u_{n-1}, t_n, z_n, x)_n \bigg) \notag\\
			& \leq \sum_{\vec t, \vec w, \vec z, \vec u} p^{|\{\orig, w_0\}|} \pp \big( F_0(\orig,w_0, u_0, z_1)\big) \notag\\ 
			& \qquad \times \prod_{i=1}^{n-1} \Big(  p^{|\{u_{i-1}, t_i\}| + 2} \pp\big( F^{(1)}(\vec v_i)\big) + p^{|\{u_{i-1},w_i\}|+|\{w_i, t_i\}| 
						+ |\{t_i,z_i,u_i\}|-3} \pp\big( F^{(2)}(\vec v_i)\big) \Big) \notag\\
			& \qquad \times p^{|\{u_{n-1},t_n\}| + |\{t_n, z_n,u_n\}|-2 }  \pp\big(F_n(u_{n-1}, t_n, z_n, u_n) \big), \label{eq:db:Pi_bound_by_F}}
where we recall that $\vec v_i = (u_{i-1}, t_i,w_i,z_i,u_i, z_{i+1})$. In the first line, $\vec t, \vec w,\vec z$ are occupied points as in Lemma~\ref{lem:db:Pi_bounds_by_F}. In both lines, $\vec u_{[0,n]} \in (\Zd)^{(n+1,1)}$, and in the second line, $\vec t, \vec w, \vec z \in (\Zd)^n$. Crucially, in the second inequality of~\eqref{eq:db:Pi_bound_by_F}, the factorization occurs due to the independence of the different percolation configurations. Moreover, it is crucial here that the number of factors of $p$ (appearing when we switch from a sum over points in $\omega$ to a sum over points in $\Zd$) depends on the number of coinciding points.

We can now decompose the $F$ events by heavy use of the BK inequality, producing bounds in terms of the $\phi$ functions introduced in Definition~\ref{def:db:psi_phi_fcts}. We start by bounding
	\al{ p^{|\{a, w\}|} \pp \big(F_0(a,w,u,z)\big) &\leq  \phi_0(a,w,u,z), \\
		 p^{|\{a,t\}|+|\{t,z,u\}|-2} \pp\big(F_n(a,t,z,x) \big) & \leq \phi_n(a,t,z,x). }
We continue to bound
	\[ p^{|\{a,t\}| + 2}  \pp\big( F^{(1)}(a,t,w,z,u,b)\big) + p^{|\{a,w\}|+|\{w,t\}| + |\{t,z,u\}|-3} \pp\big( F^{(2)}(a,t,w,z,u,b)\big) \Big) \leq \phi(a,t,w,z,u,b). \]
Plugging these bounds into~\eqref{eq:db:Pi_bound_by_F}, we obtain the new bound
	\eqq{ p \Pi_p^{(n)}(u_n) \leq \sum_{(\vec t, \vec z)_{[1,n]}, (\vec w, \vec u)_{[0,n-1]}} \hspace{-1cm}\phi_0(\orig,w_0,u_0,z_1) \phi_n(u_{n-1},t_n,z_n,x) 
								\prod_{i=1}^{n-1} \phi (u_{i-1},t_i,w_i,z_i,u_i,z_{i+1}) \label{eq:db:Pi_bound_by_phi},}
where $\vec t_{[1,n]}, \vec w_{[0,n-1]}, \vec z_{[1,n]} \in (\Zd)^n$, and $\vec u_{[0,n]}\in(\Zd)^{(n+1,1)}$. We rewrite the right-hand side of~\eqref{eq:db:Pi_bound_by_phi} by replacing the $\phi_0, \phi_n$ and $\phi$ functions by $\psi_0, \psi_n$ and $\psi$ functions. As the additional factors arising from this replacement exactly cancel out, this gives the first bound in Lemma~\ref{lem:db:pi_psi_bounds}. The observation 
	\eqq{ \psi_n(a_1,a_2,t,z,u) \leq \psi(a_1,a_2,t,a_2,z,u) \label{eq:db:Psi_last_segment_bound}}
gives the second bound and finishes the proof.
\end{proof}

We can now prove Proposition~\ref{thm:db:main_thm}:

\begin{proof}[Proof of Proposition~\ref{thm:db:main_thm}] We show that
	\eqq{ \sum_{w,u \in \Zd} \Psi^{(n)} (w,u) \leq  \tripf(\orig) \big( T_p  \big)^n, \label{eq:db:Psi_triangles_bound}}
which is sufficient due to~\eqref{eq:db:Pi_bound_by_Psi}. The proof of~\eqref{eq:db:Psi_triangles_bound} is by induction on $n$. For the base case, we bound
	\[ \col{\sum_{w,u \in \Zd} \Psi^{(0)} (w,u)=}\sum_{w,u} \psi_0(\orig,w,u) \leq \sum_{w,u} \tilde\psi_0(\orig,w,u) = p \sum_{w,u} \taupf(w) \taup(u) \taup(u-w) = \tripf(\orig) . \]
Let now $n \geq 1$. Then
	\begin{align} \sum_{w,u} \Psi^{(n)}(w,u) & = \sum_{w',u'} \Psi^{(n-1)}(w',u') \sum_{z,t,w,u\neq u'}  \psi(w',u',t,w,z,u) \notag \\
		& \leq \Big( \sum_{w',u'} \Psi^{(n-1)} (w',u') \Big) \Big( \sup_{w' \neq u'} \sum_{z,t,w,u\neq u'}  \psi(w',u',t,w,z,u) \Big). \label{eq:db:inductive_step} \end{align}
The fact that $\Psi^{(n)}(x,x) = 0$ for any $n$ and $x$ allows us to assume $w' \neq u'$ in the supremum in the second line of~\eqref{eq:db:inductive_step}. After applying the induction hypothesis, it remains to bound the second factor for $w' \neq u'$, which we rewrite as $\sup_{a \neq \orig} \sum_{t,w,z,u \neq a} \psi(\orig,a,t,w,z,u)$ by translation invariance. As it is a sum of two terms (originating from $\psi^{(1)}$ and $\psi^{(2)}$), we start with the first one and obtain
\begin{figure}
	\centering
  \includegraphics{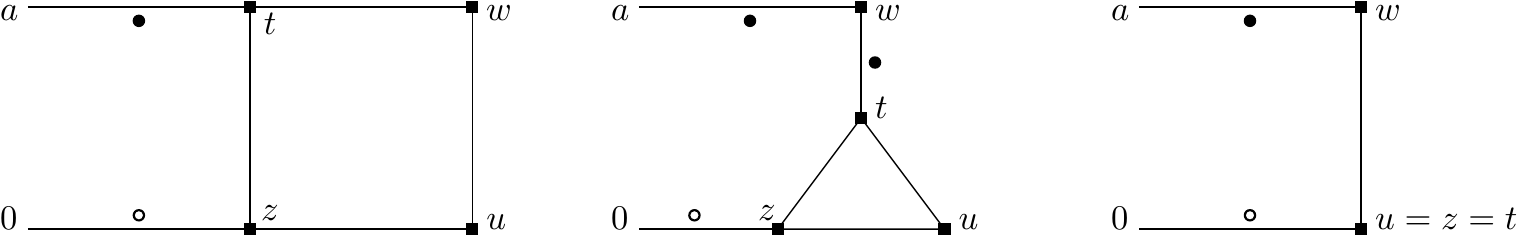}
	\caption{The pictorial representation of $\sup_{\orig \neq a}\sum_{t,w,z,u \neq a} \psi(\orig,a,t,w,z,u)$. The last two pictures represent the case distinction captured in the indicator. Points that are summed over are marked with a square.}
	\label{fig:Psi_diagrams}
\end{figure}
	\algn{ p^3  & \sum_{t,z} \Big( \taupf(t-a) \taup(z-t) \taupo(z) \Big( \sum_{u,w} \taup(w-t) \taup(u-w) \taup(z-u) \Big)\Big) \notag \\
		&\leq p \sum_{t,z} \Big( \taupf(t-a) \taup(z-t) \taupo(z) \Big(\sup_{t,z} p^2 \sum_{u,w} \taup(w-t) \taup(u-w) \taup(z-u) \Big)\Big) \leq \trip \tripof(a). \label{eq:db:psi_1_bound}}
Before treating the second term, we show how to obtain the bound from~\eqref{eq:db:psi_1_bound} pictorially, using diagrams very similar to the ones introduced in Figure~\ref{fig:Psi_diagrams}. In particular, factors of $\taup$ are represented by lines, factors of $\taupf$ and $\taupo$ by lines with an added `$\bullet$' or `$\circ$', respectively. Points summed over are represented by squares, other points (which we mostly take the supremum over, for example point $a$) are represented by colored disks. Hence, we interpret the factor $\taup(z-t)$ as a line between $t$ and $z$. Since both endpoints are summed over, we display them as squares (and without labels $t$ or $z$). We interpret the factor $\taupo(z)$ as a ($\circ$-decorated) line between $\orig$ and $z$; the origin is represented by lack of decorating the incident line. Finally, we indicate the distinctness of a pair of points (in our case $\orig \neq a$) by a disrupted two-headed arrow $\mathrel{\raisebox{-0.25 cm}{\includegraphics{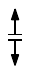}}}$. With this notation,~\eqref{eq:db:psi_1_bound} becomes
	\[ p^3 \sum \mathrel{\raisebox{-0.25 cm}{\includegraphics{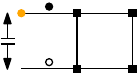}}}
			\ \leq p \sum \Big( \mathrel{\raisebox{-0.25 cm}{\includegraphics{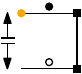}}}
				\Big(  \sup_{\textcolor{blue}{\bullet}, \textcolor{green}{\bullet}} p^2 \sum \mathrel{\raisebox{-0.25 cm}{\includegraphics{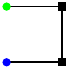}}} \Big) \Big)
			\leq \tripof \trip.\]

The second term in $\psi$, originating from $\psi^{(2)}$, contains an indicator. Resolving it splits this term into two further terms. We first consider the term arising from $|\{t,z,u\}| = 1$, which forces $w \neq t=u=z$, and the term is of the form
	\[ p \sum_{u,w} \taupo(u) \taup(w-u) \taupf(a-w) = p \sum \mathrel{\raisebox{-0.25 cm}{\includegraphics{Psi_1__decomp_1.pdf}}} \ = \tripof(a) .\]
Turning to the term due to $|\{t,z,u\}| = 3$, with a substitution of the form $y'=y-u$ for $y\in \{t,w,z\}$ in the second line, we see that
	\algn{& p^2 \sum_{t,w,z,u} \taupf(w-a) \taupo(z) \taupf(t-w) \taup(z-t)\taup(u-t) \taup(z-u) \notag \\
		=& p^2 \sum_{t', z'} \Big( \taup(z') \taup(t'-z') \taup(t')  \Big( \sum_{u,w'} \taupo(z'+u) \taupf(a-w'-u) \taupf(w-t) \Big)\Big) \notag \\
		\leq & \trip \tripoff. \label{eq:db:psi_2_bound}}
This concludes the proof. However, we also want to show how to execute the bound in~\eqref{eq:db:psi_2_bound} using diagrams. To do so, we need to represent a substitution in pictorial form. Note that after the substitution, the sum over point $u$ is w.r.t.~two factors, namely $\taupo(z'+u) \taupf(a-w'-u)$. We interpret these two factors as a line between $-u$ and $z'$ and a line between $-(u-a)$ and $w'$. In this sense, the two lines do not meet in $u$, but they have endpoints that are a constant vector $a$ apart. We represent this as
	\[ \sum_u \taupo(\textcolor{altviolet}{ z'}+u) \taupf(\textcolor{darkorange}{a}-\textcolor{blue}{w'}-u)  = \sum \mathrel{\raisebox{-0.25 cm}{\includegraphics{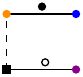}}}. \]
The bound in~\eqref{eq:db:psi_2_bound} thus becomes
	\[ p^2 \sum \mathrel{\raisebox{-0.25 cm}{\includegraphics{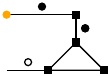}}}
			\ = p^2 \sum \mathrel{\raisebox{-0.25 cm}{\includegraphics{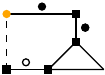}}}
			\ \leq p^2 \sum \Big( \Big( \sup_{\textcolor{blue}{\bullet}, \textcolor{green}{\bullet}} \sum
				 \mathrel{\raisebox{-0.25 cm}{\includegraphics{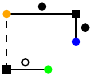}}} \Big)
				 \mathrel{\raisebox{-0.25 cm}{\includegraphics{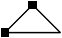}}} \Big)
			\leq \tripoff \trip,\]
where we point out that we did not use $a \neq \orig$ for the bound $\tripoff$, and so it was not indicated in the diagram.
\end{proof}

The following corollary will be needed later to show that the limit $\Pi_{p,n}$ for $n \to\infty$ exists:

\begin{corollary} \label{cor:db:Pi_pn_unsummed_bound}
For $n \geq 1$,
	\[ \sup_{x\in\Zd} \Pi_p^{(n)}(x) \leq \tripf(\orig) (1+\tripof) \big( T_p \big)^{n-1}.  \]
\end{corollary}
\begin{proof}
Note that
	\[ \Pi_p^{(n)}(x) \leq \Big( \sup_{w \neq u} \sum_{t,z} \psi_n(w,u,t,z,x) \Big) \sum_{w,u} \Psi^{(n-1)}(w,u). \]
Since we do not sum over $x$, we bound the factors depending on $x$ by $1$, and so
	\[ \sum_{t,z} \psi_n(w,u,t,z,x) \leq \taupf(x-u)\taupo(x-w) + \tripof(u-w) \leq 1 + \tripof(u-w) \]
implies the claim together with Proposition~\ref{thm:db:main_thm}.
\end{proof}

\subsection{Displacement bounds} \label{sec:db:displacement}
The aim of this section is to give bounds on $p \sum_x [1-\cos(k\cdot x)] \Pi_p^{(n)}(x)$. Such bounds are important in the analysis in Section~\ref{sec:bootstrap_analysis}. We regard $[1-\cos(k\cdot x)]$ as a ``displacement factor''. To state the main results, Propositions~\ref{thm:db:displacement_thm} and~\ref{thm:db:displacement_thm_n1}, we introduce some displacement quantities:
\begin{definition}[Diagrammatic displacement quantities] \label{def:displacement_quantities}
Let $x\in\Zd$ and $k \in\fspace$. Define
	\al{ W_p(x;k) &:= p(\taupk\ast\taupo)(x), \qquad W_p(k) = \max_{x\in\Zd} W_p(x;k), \\
		 H_p(b_1, b_2;k) &:= p^5 \hspace{-0.15cm} \sum_{t,w,z,u,v} \taup(z) \taup(t-u) \taup(t-z) \taupk(u-z) \taup(t-w) \taup(w-b_1) \taup(v-w) \taup(v+b_2-u), \\
		 H_p(k) &:=\max_{b_1\neq \orig \neq b_2 \in\Zd} H_p(b_1, b_2;k). }
\end{definition}

Note that Proposition~\ref{thm:db:bounds_for_n0} already provides displacement bounds for $n=0$. The following two results give bounds for $n\geq 1$:

\begin{prop}[Displacement bounds for $n \geq 2$]\label{thm:db:displacement_thm}
For $n \geq 2$ and $x\in\Zd$,
	 \[ p \sum_x [1-\cos(k\cdot x)] \Pi_p^{(n)}(x) \leq 11 (n+1) \big(T_p\big)^{1 \vee (n-2)} \big(\tripoff\big)^3 W_p(k) \Big[1+ \tripo+T_p + \frac{H_p(k)}{W_p(k)} \Big]. \]
\end{prop}
\begin{prop}[Displacement bounds for $n=1$]\label{thm:db:displacement_thm_n1}
For $x\in \Zd$,
	\[ p \sum_x [1-\cos(k\cdot x)] \Pi_p^{(1)}(x) \leq 9 W_p(k) \Big[\tripf(\orig)\big(\tripof+\trip\big) + \tripo + \trip \Big] + p^2 (\jeq\ast\taupk\ast\taup)(\orig).\]
\end{prop}

In preparation for the proofs, we define a function $\bar\Psi^{(n)}$, similar to $\Psi^{(n)}$, and prove an almost identical bound to the one in Proposition~\ref{thm:db:main_thm}. Let $\bar\Psi^{(0)}(t,z) = \phi_n(\orig,t,z,\orig) / \taupf(t)$. For $i \geq 1$, define
	\[ \bar\Psi^{(i)}(t,z) := \sum_{w,u,t', z'} \bar\Psi^{(i-1)}(t',z') \Big[\phi^{(1)}(\orig,t,w,z,u,z') + \phi^{(2)}(\orig,w,t,z,u,z') \Big] \frac{\taupf(t'-u)}{\taupf(t)}. \]
Note that in $\phi^{(2)}$, the points $t$ and $w$ swap roles, so that in both $\phi^{(1)}$ and $\phi^{(2)}$, $u$ is adjacent to $t'$ and $t$ is the point adjacent to $\orig$; and in particular, the factor $\taupf(t)$ cancels out. The following lemma, in combination with Lemma~\ref{lem:db:pi_psi_bounds}, is analogous to the bound~\eqref{eq:db:Psi_triangles_bound}, and so is its proof, which is omitted.
\begin{lemma} \label{lem:db:bar_psi_bound}
For $n \geq 0$,
	\[ \sum_{t,z\in\Zd} \bar\Psi^{(n)}(t,z) \leq \tripoff\big( T_p \big)^n.  \]
\end{lemma}

\begin{proof}[Proof of Proposition~\ref{thm:db:displacement_thm}]
Setting $\vec v_i = (w_{i-1}, u_{i-1}, t_i,w_i,z_i,u_i)$, we use the bound
	\eqq{ p \sum_x [1-\cos(k\cdot x)]\Pi_p^{(n)} (x) \leq \sum_x \hspace{-0.1cm}\sum_{\vec t,\vec w, \vec z, \vec u} [1-\cos(k\cdot x)] \psi_0(\orig,w_0,u_0) \psi_n(w_{n-1}, u_{n-1}, t_n, z_n, x)
				 \prod_{i=1}^{n-1} \psi(\vec v_i), \label{eq:db:disp:general_n_first_bound}}
which is, in essence, the first bound of Lemma~\ref{lem:db:pi_psi_bounds}. The next step is to distribute the displacement factor $1-\cos(k\cdot x)$ over the $n+1$ segments. To this end, we write $x = \sum_{i=0}^{n} d_i$, where $d_i =w_i-u_{i-1}$ for even $i$ and $d_i=u_i-w_{i-1}$ for odd $i$ (with the convention $u_{-1}=\orig$ and $w_n=u_n=x$). Over the course of this proof, we drop the subscript $i$ and are then confronted with a displacement $d=d_i$ (which is not to be confused with the dimension).

Using the Cosine-split lemma~\ref{lem:cosinesplitlemma}, we obtain
	\algn{p \sum_x [1-\cos(k\cdot x)]\Pi_p^{(n)} (x) & \leq (n+1) \sum_{i=0}^{n}\sum_x \hspace{-0.1cm}\sum_{\vec t,\vec w, \vec z, \vec u} [1-\cos(k\cdot d_i)] \psi_0(\orig,w_0,u_0) \notag\\
			&\hspace{2cm} \times \psi_n(w_{n-1}, u_{n-1}, t_n, z_n, x)  \prod_{j=1}^{n-1} \psi(\vec v_j), \label{eq:db:disp_d_i_split}}
with $d_i$ as introduced above. We now handle these terms for different $i$.

\underline{Case (a): $i\in\{0,n\}$.} Let us start with $i=n$, so that $d_n \in \{x-u_{n-1}, x-w_{n-1}\}$. The summand for $i=n$ in~\eqref{eq:db:disp_d_i_split} is equal to
	\algn{& \sum_{w \neq u} \Psi^{(n-1)}(w,u) \sum_{t,z,x \neq u} [1-\cos(k\cdot d')] \psi_n(w,u,t,z,x) \notag\\
		\leq & \tripf(\orig) \big(T_p\big)^{n-1} \max_{0 \neq u} \sum_{t,z,x \neq u} [1-\cos(k\cdot d)] \psi_n(\orig,u,t,z,x), \label{eq:db:disp:case_last_segment}}
where $d'\in\{x-w, x-u\}$ and $d\in\{x,x-u\}$. We expand the indicator in $\psi_n$ into two cases. If $t=z=x$, then we can bound the maximum in~\eqref{eq:db:disp:case_last_segment} by $p\sum_x [1-\cos(k\cdot d)] \taupo(x) \taup(x-u)$, which is bounded by $W_p(k)$ for both values of $d$. If $t,z,x$ are distinct points, then for $d=x$, the maximum in~\eqref{eq:db:disp:case_last_segment} becomes
	\[ p^2 \max_{u\neq 0} \sum_{t,z,x} [1-\cos(k\cdot d)] \taupo(z) \taupf(t-u) \taup(t-z) \taup(x-z)\taup(x-t)
		= p^2 \sum \mathrel{\raisebox{-0.25 cm}{\includegraphics{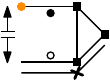}}}.\]
Note that in the pictorial representation, we represent the factor $[1-\cos(k \cdot (x-\orig))]$ by a line from $\orig$ to $x$ carrying a `$\times$' symbol. We use the Cosine-split lemma~\ref{lem:cosinesplitlemma} again to bound
	\[ [1-\cos(k\cdot x)] \leq 2 \big([1-\cos(k\cdot z)] + [1-\cos(k\cdot (x-z))]\big), \]
which results in
	\al{p^2 \sum \mathrel{\raisebox{-0.25 cm}{\includegraphics{Disp_iN_general.pdf}}} 
		\ & \leq 2 p^2 \Big[ \sum \mathrel{\raisebox{-0.25 cm}{\includegraphics{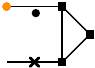}}}
			\ + \sum \mathrel{\raisebox{-0.25 cm}{\includegraphics{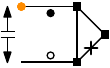}}} \Big] \\
		& \leq 2p^2 \sum \mathrel{\raisebox{-0.25 cm}{\includegraphics{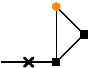}}}
			\ + 2p^3 \sum \mathrel{\raisebox{-0.25 cm}{\includegraphics{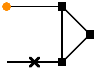}}}
			\ + 2p \sum \Big( \mathrel{\raisebox{-0.25 cm}{\includegraphics{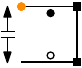}}} 
			 \Big( \sup_{\textcolor{blue}{\bullet}, \textcolor{green}{\bullet}} p \sum \mathrel{\raisebox{-0.25 cm}{\includegraphics{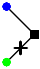}}}
			 				 \Big)\Big) \\
		& \leq 2p \sum \Big( \mathrel{\raisebox{-0.25 cm}{\includegraphics{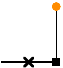}}} 
			 \Big( \sup_{\textcolor{blue}{\bullet}, \textcolor{green}{\bullet}} p \sum \mathrel{\raisebox{-0.25 cm}{\includegraphics{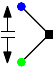}}} 
			 				\Big)\Big)
			 \ + 2p^3  \sum \mathrel{\raisebox{-0.25 cm}{\includegraphics{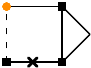}}} + 2 \tripof W_p(k) \\
		& \leq 2\tripf W_p(k) + 2p^2 \sum \Big( \Big( \sup_{\textcolor{blue}{\bullet}, \textcolor{green}{\bullet}} p \sum
			\mathrel{\raisebox{-0.25 cm}{\includegraphics{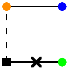}}} \Big)
			  \mathrel{\raisebox{-0.25 cm}{\includegraphics{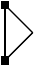}}} \Big) + 2 \tripof W_p(k)\\
		& \leq 2(2\tripof+\trip) W_p(k). }
It is not hard to see that a displacement $d=x-u$ yields the same bound.
Similar computations show that the case $i=0$ yields a contribution of at most
	\[  \tripoff \big(T_p\big)^{n-1} \tripof W_p(k). \]
	
\underline{Case (b): $1 \leq i < n$.} We want to apply both the bound~\eqref{eq:db:Psi_triangles_bound} and Lemma~\ref{lem:db:bar_psi_bound}. To this end, we rewrite the $i$-th summand in~\eqref{eq:db:disp_d_i_split} as
	\al{ \sum_x \hspace{-0.1cm} & \sum_{\vec t,\vec w, \vec z, \vec u} [1-\cos(k\cdot d_i)] \psi_0(\orig,w_0,u_0) \psi_n(w_{n-1}, u_{n-1}, t_n, z_n, x)  \prod_{j=1}^{n-1} \psi(\vec v_j) \\
		& = \sum_x \sum_{a_1,a_2,b_1,b_2} \Big( \Psi^{(i-1)}(a_1,a_2) \bar\Psi^{(n-i-1)}(b_1-x,b_2-x) \\
		& \qquad\qquad\qquad \times \sum_{t,w,z,u} \underbrace{\phi(a_2,t,w,z,u,b_2) [1-\cos(k\cdot d_i)] \taupo(z-a_1)\taupf(b_1-u)}_{=:\tilde\phi(a_1,a_2,t,w,z,u,b_1, b_2;k,d)} \Big) \\
		& \leq \big(\tripf(\orig)\big) \big(T_p\big)^{i-1} \sum_{b_1',b_2'} \Big( \bar\Psi^{(n-i-1)}(b_1',b_2') \max_{a_1 \neq a_2} 
					\sum_{t,w,z,u,x} \tilde\phi(a_1,a_2,t,w,z,u,b_1'+x, b_2'+x;k,d_i) \Big) \\
		& \leq \big(\tripf(\orig) \tripoff\big) \big( T_p\big)^{n-2} \max_{a_1 \neq a_2,b_1\neq b_2} \sum_{t,w,z,u,x} \tilde\phi(a_1,a_2,t,w,z,u,b_1+x, b_2+x;k,d_i) \\
		& \leq \big(\tripoff\big)^2 \big( T_p\big)^{n-2} \max_{\orig \neq a,\orig \neq b} \sum_{t,w,z,u,x} \tilde\phi(\orig,a,t,w,z,u,b+x,x;k,d_i),  }
where we use the substitution $b_j'=x-b_j$ in the second line and the bound $\tripf(\orig) \leq \tripoff$ in the last line. It remains to bound the sum over $\tilde\phi$. We first handle the term due to $\phi^{(1)}$, and we call it $\tilde\phi^{(1)}$. Depending on the orientation of the diagram (i.e., the parity of $i$), the displacement $d=d_i$ is either $d=w-a=(w-t)+(t-a)$ or $d=u=(u-z)+z$. We perform the bound for $d=u$ and use the Cosine-split lemma~\ref{lem:cosinesplitlemma} once, so that we now have a displacement on an actual edge. In pictorial bounds, abbreviating $\vec v=(\orig,a,t,w,z,u,b+x,x;k,u)$, this yields
	\begin{align}\sum_{t,w,z,u,x} \tilde\phi^{(1)}(\vec v) &= p^3 \sum \mathrel{\raisebox{-0.25 cm}{\includegraphics{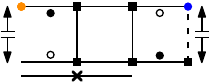}}}
			 \ \leq  2p^3 \Big[ \sum \mathrel{\raisebox{-0.25 cm}{\includegraphics{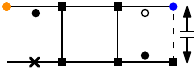}}}
				\ + \sum \mathrel{\raisebox{-0.25 cm}{\includegraphics{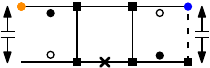}}} \Big]   \label{eq:db_dsp:psi1} \\
		& = 2p^3 \Big[ \sum \mathrel{\raisebox{-0.25 cm}{\includegraphics{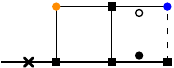}}}
			\ + p \sum \mathrel{\raisebox{-0.25 cm}{\includegraphics{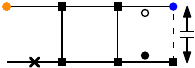}}}
			\ + \sum \Big( \mathrel{\raisebox{-0.25 cm}{\includegraphics{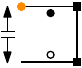}}}
					\Big( \sup_{\textcolor{orange}{\bullet}, \textcolor{green}{\bullet}}
					\sum \mathrel{\raisebox{-0.25 cm}{\includegraphics{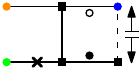}}} \Big)\Big) \Big]. \notag \end{align}
The bound in~\eqref{eq:db_dsp:psi1} consists of three summands. The first is
	\[ 2p^3 \sum \mathrel{\raisebox{-0.25 cm}{\includegraphics{Disp_i_1__split1_coll.pdf}}}
			\ \leq 2p \sum \Big( \mathrel{\raisebox{-0.25 cm}{\includegraphics{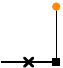}}}
				\Big( \sup_{\textcolor{altviolet}{\bullet}, \textcolor{turquoise}{\bullet}} p^2 \sum \mathrel{\raisebox{-0.25 cm}{\includegraphics{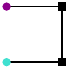}}}
				\Big( \sup_{\textcolor{darkorange}{\bullet}, \textcolor{green}{\bullet}} p^2 \sum \mathrel{\raisebox{-0.25 cm}{\includegraphics{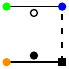}}}
			\Big) \Big) \Big) \leq 2 \tripoff \trip W_p(k), \]
the second is 
	\al{ 2p^4 \sum \sum \mathrel{\raisebox{-0.25 cm}{\includegraphics{Disp_i_1__split1_nocoll.pdf}}}
			\ & = 2p^4 \sum \mathrel{\raisebox{-0.25 cm}{\includegraphics{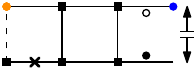}}}
			\ \leq 2p \sum \Big( \Big( \Big( \sup_{\textcolor{altviolet}{\bullet}, \textcolor{green}{\bullet}} 
				p \sum \mathrel{\raisebox{-0.25 cm}{\includegraphics{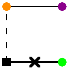}}} \Big)
				\sup_{\textcolor{altviolet}{\bullet}, \textcolor{green}{\bullet}} p^2 \sum \mathrel{\raisebox{-0.25 cm}{\includegraphics{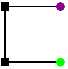}}} \Big)
				\mathrel{\raisebox{-0.25 cm}{\includegraphics{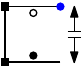}}} \Big) \\
		& \leq \tripof \trip W_p(k), }
and the third is
	\al{ 2p^3 \sum \Big( \mathrel{\raisebox{-0.25 cm}{\includegraphics{Disp_i_1__split2_bound1.pdf}}}
					& \Big( \sup_{\textcolor{orange}{\bullet}, \textcolor{green}{\bullet}} \sum \mathrel{\raisebox{-0.25 cm}{\includegraphics{Disp_i_1__split2_bound2.pdf}}} \Big)\Big)
			\ \leq 2p^2 \tripof \sup_{\textcolor{orange}{\bullet}} \sum \mathrel{\raisebox{-0.25 cm}{\includegraphics{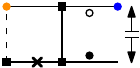}}} \\
			& \leq 2p^2 \tripof \sup_{\textcolor{orange}{\bullet}} \sum \Big( \sup_{\textcolor{altviolet}{\bullet}, \textcolor{green}{\bullet}}
				 \sum \mathrel{\raisebox{-0.25 cm}{\includegraphics{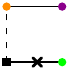}}} \Big)
				 \mathrel{\raisebox{-0.25 cm}{\includegraphics{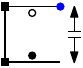}}} \Big)
			\leq 2 \big( \tripof \big)^2 W_p(k). }
The displacement $d=w-a$ satisfies the same bound. In total, the contribution in $\tilde\phi$ due to $\phi^{(1)}$ is at most
	\[ 4 \big(\tripoff\big)^3 \big(T_p \big)^{n-2} (\tripof + \trip) W_p(k). \]
Let us now tend to $\tilde\phi^{(2)}$. To this end, we first write $\tilde\phi^{(2)} = \sum_{j=3}^{5} \tilde\phi^{(j)}$, where
	\al{\tilde\phi^{(3)}(\orig,a,t,w,z,u,b+x,x;k,d) &= \tilde\phi^{(2)}(\orig,a,t,w,z,u,b+x,x;k,d) \mathds 1_{\{|\{t,z,u\}|=3\}}, \\
		\tilde\phi^{(4)}(\orig,a,t,w,z,u,b+x,x;k,d) &= [1-\cos(k\cdot d)] \delta_{z,u} \delta_{t,u} \taupo(u) \taup(w-u) \taupf(a-w) \taupf(u+x) \taupo(b-w-x), \\
		\tilde\phi^{(5)}(\orig,a,t,w,z,u,b+x,x;k,d) &= [1-\cos(k\cdot d)] \delta_{z,u} \delta_{t,u} \delta_{a,w} \taupo(u) \taup(a-u) \taupf(u+x) \taupo(b-a-x).	} 
Again, we set $\vec v=(\orig,a,t,w,z,u,b+x,x;k,u)$. Then
	\begin{align}\sum_{t,w,z,u,x} \tilde\phi^{(3)}(\vec v) &= p^2 \sum \mathrel{\raisebox{-0.25 cm}{\includegraphics{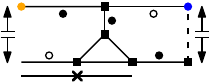}}} \notag \\
			 & \leq  p^2 \sum \mathrel{\raisebox{-0.25 cm}{\includegraphics{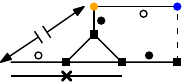}}} 
			 		\ + 2p^3 \Big[ \sum \mathrel{\raisebox{-0.25 cm}{\includegraphics{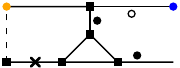}}}
			 		\ + \sum \mathrel{\raisebox{-0.25 cm}{\includegraphics{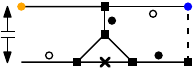}}} \Big]. \label{eq:db_dsp:psi3}\end{align}
The first term in~\eqref{eq:db_dsp:psi3} is
	\al{p^2 \sum \mathrel{\raisebox{-0.25 cm}{\includegraphics{Disp_i_3__coll1.pdf}}} &\leq p^2 
				\sum \Big( \mathrel{\raisebox{-0.25 cm}{\includegraphics{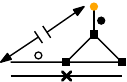}}} \Big( \sup_{\textcolor{altviolet}{\bullet}, \textcolor{green}{\bullet}}
				\sum \mathrel{\raisebox{-0.25 cm}{\includegraphics{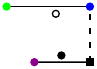}}}\Big)\Big)
			 \leq \tripoff p^2 \sum \mathrel{\raisebox{-0.25 cm}{\includegraphics{Disp_i_3__coll1_bound1.pdf}}} \\
		& \leq 2 \tripoff p^2 \Big[ \sum \mathrel{\raisebox{-0.25 cm}{\includegraphics{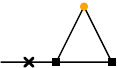}}}
				\ + p \sum \mathrel{\raisebox{-0.25 cm}{\includegraphics{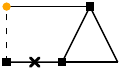}}}
				\ + \sum \Big( \mathrel{\raisebox{-0.25 cm}{\includegraphics{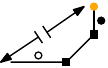}}} \Big(
				  \sup_{\textcolor{blue}{\bullet}, \textcolor{green}{\bullet}} 
				  	\sum \mathrel{\raisebox{-0.25 cm}{\includegraphics{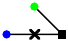}}}\Big)\Big) \Big] \\
		& \leq 2\tripoff \Big[ p^2 \sum \Big( \mathrel{\raisebox{-0.25 cm}{\includegraphics{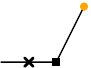}}}
					\Big( \sup_{\textcolor{blue}{\bullet}, \textcolor{green}{\bullet}} 
					\sum \mathrel{\raisebox{-0.25 cm}{\includegraphics{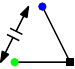}}} \Big)\Big) \\
		& \hspace{2.2cm} + p^3 \sum \Big(\Big( \sup_{\textcolor{blue}{\bullet}, \textcolor{green}{\bullet}} 
					\sum \mathrel{\raisebox{-0.25 cm}{\includegraphics{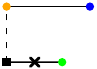}}} \Big)
					\mathrel{\raisebox{-0.25 cm}{\includegraphics{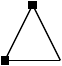}}} \Big) \Big] + \tripof W_p(k) \\
		& \leq  2\tripoff W_p(k) \Big( \tripof + \tripf + \trip \Big), }
the second term is
	\[  2p^3 \sum \mathrel{\raisebox{-0.25 cm}{\includegraphics{Disp_i_3__nocoll1_split1_subst.pdf}}}
			\ \leq 2p^3 \sum \Big( \Big( \sup_{\textcolor{altviolet}{\bullet}, \textcolor{green}{\bullet}}
				 \sum \mathrel{\raisebox{-0.25 cm}{\includegraphics{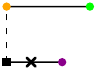}}} \Big)
				 \mathrel{\raisebox{-0.25 cm}{\includegraphics{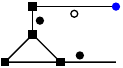}}} \Big)
			\leq 2 \tripoff \trip W_p(k), \]
and the third term is
	\begin{align} 2p^3 \sum \mathrel{\raisebox{-0.25 cm}{\includegraphics{Disp_i_3__nocoll1_split2.pdf}}} &=
				2p^3 \Big[ \sum \mathrel{\raisebox{-0.25 cm}{\includegraphics{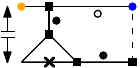}}}
				\ + p \sum \mathrel{\raisebox{-0.25 cm}{\includegraphics{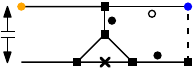}}} \Big] \notag\\
		& \leq  2p^3 \Big[ \sum \mathrel{\raisebox{-0.25 cm}{\includegraphics{Disp_i_3__nocoll1_split2_coll2.pdf}}}
				\ + \sum \Big( \mathrel{\raisebox{-0.25 cm}{\includegraphics{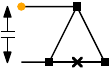}}} \Big(
				\sup_{\textcolor{altviolet}{\bullet}, \textcolor{green}{\bullet}} 
					\sum \mathrel{\raisebox{-0.25 cm}{\includegraphics{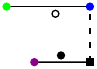}}} \Big)\Big)
				+ p \sum  \mathrel{\raisebox{-0.25 cm}{\includegraphics{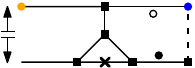}}} \Big] \notag\\
		& \leq 2 \tripoff\trip W_p(k) + 2p^3 \Big[ \sum \Big( \mathrel{\raisebox{-0.25 cm}{\includegraphics{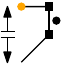}}}
				 \Big( \sup_{\textcolor{altviolet}{\bullet}, \textcolor{turquoise}{\bullet}} 
				 	\sum \mathrel{\raisebox{-0.25 cm}{\includegraphics{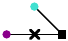}}}
				 \Big( \sup_{\textcolor{darkorange}{\bullet}, \textcolor{green}{\bullet}} 
				 	\sum \mathrel{\raisebox{-0.25 cm}{\includegraphics{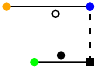}}} \Big)\Big)\Big) \notag\\
		& \hspace{5cm} +  p \sum \mathrel{\raisebox{-0.25 cm}{\includegraphics{Disp_i_3__nocoll1_split2_nocoll2_nocoll3.pdf}}} \Big] \notag\\
		& \leq 2 \tripoff W_p(k) \big( \trip + \tripf \big) + 2p^4 \sum \mathrel{\raisebox{-0.25 cm}{\includegraphics{Disp_i_3__nocoll1_split2_nocoll2_nocoll3.pdf}}}.
						\label{eq:db_dsp:psi3_almost_H} \end{align}
We are left to handle the last diagram appearing in the last bound of~\eqref{eq:db_dsp:psi3_almost_H}, which contains one factor $\taupo$ and one factor $\taupf$. We distinguish the case where neither collapses (this leads to the diagram $H_p(k)$) and the case where are least one of the factors collapses. Using $\taupf \leq \taupo$ and the substitution $y'=y-u$ for $y\in\{w,z,t\}$, we obtain
	\al{2p^4 \sum \mathrel{\raisebox{-0.25 cm}{\includegraphics{Disp_i_3__nocoll1_split2_nocoll2_nocoll3.pdf}}} 
		& \leq 2 H_p(k) + 4p^4 \sum_{z,t,w,u} \taup(t-u) \taup(w-t) \taupo(u+a_2-w) \\
		& \hspace{5cm} \times \taupk(z-u) \taup(t-z) \taup(z) \taup(w-a_1)  \\
		& = 2H_p(k) + 4 p^4 \sum_{t',w'} \Big( \taup(t') \taup(w'-t') \taupo(a_2-w') \sum_z \Big( \taupk(z') \taup(t'-z') \\
		& \hspace{5cm} \times \sum_u  \taup(z'+u) \taup(a_1-w'-u) \Big)\Big) \\
		& \leq 2H_p(k) + 4 \tripf(\orig) \tripo W_p(k). }
In total, this yields an upper bound on~\eqref{eq:db_dsp:psi3} of the form
	\[6 (\tripoff)^3 (T_p)^{n-2} \big[(\tripof + \trip + \tripo) W_p(k) + H_p(k) \big]. \]
The same bound is good enough for the displacement $d=w-a$. Turning to $j=4$, we consider the displacement $d=u$ and see that
	\al{\sum_{t,w,z,u,x} \tilde\phi^{(4)}(\vec v) &= p^2 \sum \mathrel{\raisebox{-0.25 cm}{\includegraphics{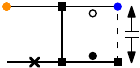}}} 
			\ = p^2 \sum \mathrel{\raisebox{-0.25 cm}{\includegraphics{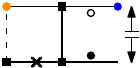}}} 
				\ \leq	p \sum \Big( \Big(  \sup_{\textcolor{altviolet}{\bullet}, \textcolor{green}{\bullet}} 
				p \sum \mathrel{\raisebox{-0.25 cm}{\includegraphics{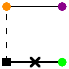}}} \Big)
				\mathrel{\raisebox{-0.25 cm}{\includegraphics{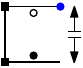}}} \Big) \\
		& \leq \tripof W_p(k),  }
which is also satisfied for $d=w-a$. Finally, $j=5$ forces $d=u$, and we have
	\[ \sum_{t,w,z,u,x} \tilde\phi^{(5)}(\vec v) = p^2 \sum \mathrel{\raisebox{-0.25 cm}{\includegraphics{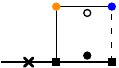}}} 
				\ \leq	p \sum \Big( \Big(  \sup_{\textcolor{altviolet}{\bullet}, \textcolor{green}{\bullet}} 
				p \sum \mathrel{\raisebox{-0.25 cm}{\includegraphics{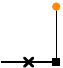}}} \Big)
				\mathrel{\raisebox{-0.25 cm}{\includegraphics{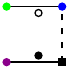}}} \Big)
		\leq \tripoff W_p(k),  \]
and we see that this bound is not good enough for $n=2$. To get a better bound for $n=2$, we bound
	\al{ p \sum_{w,u,s,t,z,x} & \tilde\psi_0(\orig,w,u) \taupk(s-w) \taup(s-u) \psi_n(u,s,t,z,x) \
			\leq \Big( p^2 \sum \mathrel{\raisebox{-0.25 cm}{\includegraphics{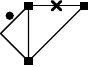}}} \Big) \sup_{u \neq s} \sum_{t,z,x} \psi_n(u,s,t,z,x) \\
		& \leq  \tripf(\orig) W_p(k) T_p, }
where we recall that $\tilde\psi_0$ is an upper bound on $\psi_0$ (see Definition~\ref{def:db:psi_phi_fcts}). The above bound is due to the fact that, thanks to~\eqref{eq:db:Psi_last_segment_bound}, the supremum over the sum over $\psi_n$ is bounded by the supremum in~\eqref{eq:db:inductive_step}. 
\end{proof}
\begin{proof}[Proof of Proposition~\ref{thm:db:displacement_thm_n1}]
Let $n=1$. Expanding the two cases in the indicator of $\phi_n$ gives
	\algn{ p \sum_x [1-\cos(k\cdot x)] \Pi_p^{(1)}(x) &\leq \sum_{w,u,t,z,x} [1-\cos(k\cdot x)] \phi_0(\orig,w,u,z) \phi_n(u,t,z,x) \notag\\
		& \leq p^2 \sum_{w,u,t,z,x} [1-\cos(k\cdot x)] \tilde\phi_0(\orig,w,u,z) \taupf(t-u) \taup(z-t) \taup(z-x) \taup(t-x) \notag\\
		& \quad + p  \sum_{w,u,x} [1-\cos(k\cdot x)] \phi_0(\orig,w,u,z) \taup(x-u) \label{eq:db:disp_n1_first_bound}, }
where we used the bound $\phi_0 \leq \tilde\phi_0$ (see Definition~\ref{def:db:psi_phi_fcts}) for the first summand. Since $\phi_0$ is a sum of two terms,~\eqref{eq:db:disp_n1_first_bound} is equal to
	\algn{& p^3 \sum_{w,u,t,z,x}[1-\cos(k\cdot x)] \taupf(w) \taup(u) \taup(u-w) \taupo(z-w) \taupf(t-u) \taup(z-t) \taup(x-z) \taup(x-t) \notag\\
		+& p^2 \sum_{u,x} \jeq(u) \taupk(x) \taup(x-u) \notag\\
		+& p^2 \sum_{w,u,x} [1-\cos(k\cdot x)] \taupf(w) \ttaup(u) \taup(u-w) \taupo(x-w) \taup(x-u). \label{eq:db:disp_n1_second_bound}}
We use the Cosine-split lemma~\ref{lem:cosinesplitlemma} on the first term of~\eqref{eq:db:disp_n1_second_bound} to decompose $x=u + (z-u) + (x-z)$, which gives
	\al{p^3 \sum_{w,u,t,z,x} & [1-\cos(k\cdot x)] \taupf(w) \taup(u) \taup(u-w) \taupo(z-w) \taupf(t-u) \taup(z-t) \taup(x-u) \taup(x-t) \\
		& = p^3 \sum \mathrel{\raisebox{-0.25 cm}{\includegraphics{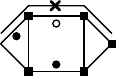}}}
			\ \leq 3 p^3 \Big[p \sum \mathrel{\raisebox{-0.25 cm}{\includegraphics{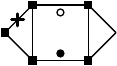}}}
				\ + \sum \mathrel{\raisebox{-0.25 cm}{\includegraphics{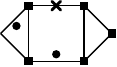}}}
				\ + \sum \mathrel{\raisebox{-0.25 cm}{\includegraphics{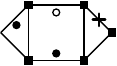}}} \Big]  \\
		& \leq 3p^2 \sum \Big(\Big( \sup_{\textcolor{altviolet}{\bullet}, \textcolor{green}{\bullet}} p \sum \Big( \sup_{\textcolor{darkorange}{\bullet}, \textcolor{blue}{\bullet}}
			p\sum  \mathrel{\raisebox{-0.25 cm}{\includegraphics{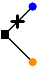}}} \Big)
				\mathrel{\raisebox{-0.25 cm}{\includegraphics{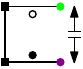}}} \Big)
				\mathrel{\raisebox{-0.25 cm}{\includegraphics{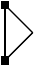}}} \Big)
			+  3p^3 \sum \mathrel{\raisebox{-0.25 cm}{\includegraphics{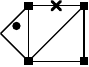}}} \\
		& \quad  + 3p^4 \sum \Big( \mathrel{\raisebox{-0.25 cm}{\includegraphics{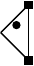}}}
			\Big( \sup_{\textcolor{darkorange}{\bullet}, \textcolor{blue}{\bullet}} \sum \mathrel{\raisebox{-0.25 cm}{\includegraphics{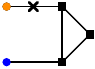}}} \Big) \Big)
			\ + 3p \sum \Big( \mathrel{\raisebox{-0.25 cm}{\includegraphics{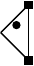}}}
				\Big(\sup_{\textcolor{darkorange}{\bullet}, \textcolor{blue}{\bullet}} p \sum  \mathrel{\raisebox{-0.25 cm}{\includegraphics{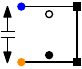}}} 
				\Big( \sup_{\textcolor{altviolet}{\bullet}, \textcolor{green}{\bullet}} p \sum \mathrel{\raisebox{-0.25 cm}{\includegraphics{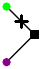}}} \Big)\Big)\Big) \\
		& \leq 3 W_p(k) \tripof \trip +  3p^2 \sum\Big( \mathrel{\raisebox{-0.25 cm}{\includegraphics{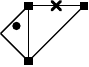}}}
				\Big( \sup_{\textcolor{darkorange}{\bullet}, \textcolor{blue}{\bullet}} p\sum \mathrel{\raisebox{-0.25 cm}{\includegraphics{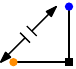}}} \Big)\Big) \\
		& \quad + 3p^3 \tripf(\orig) \sup_{\textcolor{darkorange}{\bullet}} \sum \mathrel{\raisebox{-0.25 cm}{\includegraphics{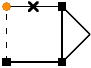}}} 
			 + 3 W_p(k) \tripof \tripf(\orig) \\
		& \leq 3 W_p(k) \tripf(\orig) \big( 3\tripof \big) + 3p^3 \tripf(\orig) \sup_{\textcolor{darkorange}{\bullet}} \sum \Big(\Big( \sup_{\textcolor{blue}{\bullet},\textcolor{green}{\bullet}}
				\sum \mathrel{\raisebox{-0.25 cm}{\includegraphics{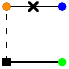}}} \Big)
				\mathrel{\raisebox{-0.25 cm}{\includegraphics{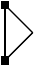}}} \Big) \\
		& \leq 3 W_p(k) \tripf(\orig) \big( 3\tripof + \trip \big).  }
The second term in~\eqref{eq:db:disp_n1_second_bound} is $p^2 (\jeq\ast\taupk\ast \taup)(\orig)$. Depicting the factor $\ttaup$ as a disrupted line $\mathrel{{\includegraphics{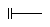}}}$, the third term in~\eqref{eq:db:disp_n1_second_bound} is
	\al{p^2 \sum_{w,u,x} & [1-\cos(k\cdot x)] \taupf(w) \ttaup(u) \taup(u-w) \taupo(x-w) \taup(x-u)  = p^2 \sum \mathrel{\raisebox{-0.25 cm}{\includegraphics{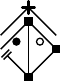}}} \\
		& \leq 2p^3 \sum \mathrel{\raisebox{-0.25 cm}{\includegraphics{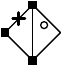}}}
			\ + 2p^2 \sum \mathrel{\raisebox{-0.25 cm}{\includegraphics{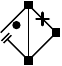}}} 
		 	\ \leq 2 W_p(k) \Big[ \tripo(\orig) + p \big((\delta_{\orig,\cdot} + p \taup)\ast\taup\ast \ttaup\big)(\orig) \Big] \\
		& \leq 2 W_p(k) \big[ \tripo + p (\taup\ast\ttaup)(\orig) + p^2 (\taup^{\ast 2}\ast\ttaup)(\orig) \big] \leq 2 W_p(k) \big[ \tripo + 2 \trip \big].}
In the above, we have used that $\ttaup(x) \leq \taup(x)$ as well as $\ttaup(x) \leq p (\jeq\ast\taup)(x) \leq p\taup^{\ast 2}(x)$.
\end{proof}

\section{Bootstrap analysis} \label{sec:bootstrap_analysis}
\subsection{Introduction of the bootstrap functions}\label{sec:boot:intro}
This section brings the previous results together to prove Proposition~\ref{thm:convergence_of_LE}, from which Theorem~\ref{thm:main_theorem_triangle_condition} follows with little extra effort. The remaining strategy of proof is standard and described in detail in~\cite{HeyHof17}. In short, it is the following: We introduce the bootstrap function $f$ in~\eqref{eq:boot:intro:f_functions_def}. In Section~\ref{sec:boot:consequences}, and in particular in Proposition~\ref{thm:convergence_of_LE}, we prove several bounds in terms of $f$, including bounds uniform in $p \in [0,p_c)$ under the additional assumption that $f$ is uniformly bounded.

In Section~\ref{sec:boot:bootstrap_argument}, we show that $f(0) \leq 3$ and that $f$ is continuous on $[0,p_c)$. Lastly, we show that on $[0,p_c)$, the bound $f\leq 4$ implies $f\leq 3$. This is called the improvement of the bounds, and it is shown by employing the implications from Section~\ref{sec:boot:consequences}. As a consequence of this, the results from Section~\ref{sec:boot:consequences} indeed hold uniformly in $p\in[0,p_c)$, and we may extend them to $p_c$ by a limiting argument.

Let us recall the notation $\taupk(x) = [1-\cos(k\cdot x)] \taup(x), \jek(x) = [1-\cos(k\cdot x)] \jeq(x)$. We extend this to $\connf_k(x) = [1-\cos(k\cdot x)] \connf(x)$. We note that $\chi(p)$ was defined as $\chi(p) = \E[|\C(\orig)|]$ and that $\chi(p) = 1 + p \sum_{x \in \Zd} \taup(x)$. We define
	\[ \lamp = 1 - \frac{1}{\chi(p)} = 1 - \frac{1}{1+p \ftau(0)} .\]
We define the bootstrap function $f= f_1 \vee f_2 \vee f_3$ with
	\eqq{f_1(p) = 2dp, \qquad f_2(p) = \sup_{k \in\fspace} \frac{p |\ftau(k)| }{ \fgmu(k)}, \qquad f_3(p) = \sup_{k,l \in\fspace} \frac{p|\ftaupk(l)| }{ \ulam(k,l)}, \label{eq:boot:intro:f_functions_def}}
where $\ulam$ is defined as
	\[\ulam(k,l) := 3000 [1-\fconnf(k)] \Big( \fgreenslam(l-k) \fgreenslam(l) + \fgreenslam(l)\fgreenslam(l+k) + \fgreenslam(l-k)\fgreenslam(l+k) \Big). \]
We note that $\ftaupk$ relates to $\Delta_k \ftau$, the discretized second derivative of $\ftau$, as follows:
	\[\Delta_k \ftau(l) := \ftau(l-k) + \ftau(l+k) - 2 \ftau(l) = -2 \ftaupk(l) .\]
The following result bounds the discretized second derivative of the random walk Green's function:
\begin{lemma}[Bounds on $\Delta_k$, \cite{Sla06}, Lemma 5.7] \label{lem:bootstrap:Delta_k_Ulam_bound}
Let $a(x) = a(-x)$ for all $x\in\Zd$, set $\widehat A(k) = (1-\widehat a(k))^{-1}$, and let $k,l\in\fspace$. Then
	\al{|\Delta_k \widehat A(l) | \leq \big(\widehat{|a|} (0) - \widehat{|a|}(k) \big) \times & \Big( \big[ \widehat A(l-k) + \widehat A(l+k) \big] \widehat A(l)  \\
		& \quad + 8 \widehat A(l-k) \widehat A(l+k)  \widehat A(l) \big[ \widehat{|a|}(0) - \widehat{|a|}(l) \big]\Big). }
In particular,
	\[ |\Delta_k \fgreenslam(l)| \leq [1-\fconnf(k)] \Big( \fgreenslam(l)\fgreenslam(l-k) + \fgreenslam(l)\fgreenslam(l+k) + 8 \fgreenslam(l-k)\fgreenslam(l+k) \Big) .\]
\end{lemma}
A natural first guess for $f_3$ might have been $\sup p |\Delta_k\ftau(l)| / | \Delta_k \fgmu(l)|$. However, $\Delta_k \fgmu(l)$ may have roots, which makes this guess an inconvenient choice for $f_3$. In contrast, $\ulam(k,l) >0$ for $k \neq 0$. Hence, the bound in Lemma~\ref{lem:bootstrap:Delta_k_Ulam_bound} supports the idea that $f_3$ is a reasonable definition.

\subsection{\col{Consequences of the diagrammatic bounds}} \label{sec:boot:consequences} The main result of this section, and a crucial result in this paper, is Proposition~\ref{thm:convergence_of_LE}. Proposition~\ref{thm:convergence_of_LE} proves (in high dimension) the convergence of the lace expansion derived in Proposition~\ref{thm:lace_expansion} by giving bounds on the lace-expansion coefficients. Under the additional assumption that $f\leq 4$ on $[0,p_c)$, these bounds are uniform in $p\in [0,p_c)$.

\begin{prop}[Convergence of the lace expansion and Ornstein-Zernike equation] \label{thm:convergence_of_LE} \ 
\begin{enumerate}
\item \label{firstassertion} Let $n\in\N_0$ and $p \in [0,p_c)$. Then there is $d_0 \geq 6$ and a constant $c_f=c(f(p))$ (increasing in $f$ and independent of $d$) such that, for all $d > d_0$,
	\algn{\sum_{x\in\Zd} p | \Pi_{p,n}(x)| \leq c_f/d, & \qquad \sum_{x\in\Zd} [1-\cos(k\cdot x)\big] p | \Pi_{p,n}(x)| \leq [1-\fconnf(k)] c_f/d, \label{eq:boot:convergence_of_Pi_M} \\
		\sup_{x\in\Zd} & p \sum_{m=0}^{n} | \Pi_p^{(m)}(x)|  \leq c_f, \label{eq:boot:convergence_of_Pi_n_unsummed}}
and
	\eqq{ \sum_{x\in\Zd} |R_{p,n}(x)| \leq c_f (c_f/d)^n \ftau(0). \label{eq:boot:conseq:R_bounds}}
Consequently, $\Pi_p :=\lim_{n\to\infty} \Pi_{p,n}$ is well defined and $\taup$ satisfies the Ornstein-Zernike equation (OZE), taking the form
	\eqq{ \taup(x) = \jeq(x) + \Pi_p(x) + p\big((\jeq+\Pi_p)\ast\taup\big)(x). \label{eq:boot:conseq:OZE}}
\item \label{secondassertion} 
Let $f \leq 4$ on $[0,p_c)$. Then there is a constant $c$ and $d_0 \geq 6$ such that the bounds~\eqref{eq:boot:convergence_of_Pi_M},~\eqref{eq:boot:convergence_of_Pi_n_unsummed},~\eqref{eq:boot:conseq:R_bounds} hold for all $d > d_0$ with $c_f$ replaced by $c$ for all $p \in [0,p_c)$. Moreover, the OZE~\eqref{eq:boot:conseq:OZE} holds.
\end{enumerate}
\end{prop}

A standard assumption in the lace expansion literature is a bound on $f(p)$ (often $f(p)\le 4$ as in part {\em\ref{secondassertion}}\ of the proposition), and then it is shown that this implies $f(p) \leq 3$. This is part of the so-called bootstrap argument. 

We first formulate part~{\em\ref{firstassertion}}~of Proposition~\ref{thm:convergence_of_LE} (of which part~\emph{\ref{secondassertion}} follows straightforwardly) demonstrating that the bootstrap argument is not necessary to obtain convergence of the lace expansion and thus establish the OZE for a \emph{fixed} value $p<p_c$ provided that the dimension is large enough. However, without uniformity in $p$, $d_0$ might depend on $p$ and diverge as $p \nearrow p_c$. Hence, this approach alone does not allow to extend the results to $p_c$. 

It is at this point that the bootstrap argument (and thus Section~\ref{sec:boot:bootstrap_argument}) comes into play. In Section~\ref{sec:boot:bootstrap_argument}, we indeed prove that $f \leq 4$ and so the second part of Proposition~\ref{thm:convergence_of_LE} applies. This is instrumental in proving Theorem \ref{thm:main_theorem_triangle_condition}. We get the following corollary:

\begin{corollary}[OZE at $p_c$] \label{cor:boot:OZE_at_critical_point}
There is $d_0$ such that for all $d>d_0$, the limit $\Pi_{p_c} = \lim_{p \nearrow p_c} \Pi_p$ exists and is given by $\Pi_{p_c}= \sum_{n \geq 0} (-1)^n \Pi_{p_c}^{(n)}$, where $\Pi_{p_c}^{(n)}$ is the extension of Definition~\ref{def:le:lace_expansion_coefficients} at $p=p_c$. Consequently, the bounds in Proposition~\ref{thm:convergence_of_LE} and the OZE~\eqref{eq:boot:conseq:OZE} extend to $p_c$.
\end{corollary}

Proposition~\ref{thm:convergence_of_LE} follows without too much effort as a consequence of Lemmas~\ref{lem:bootstrap:bounds_on_trip},~\ref{lem:bootstrap:bounds_on_W},~\ref{lem:bootstrap:bounds_on_Pi0}, and~\ref{lem:bootstrap:bounds_on_H}. \col{Part of the lace expansion's general strategy of proof in the bootstrap analysis is to use the Inverse Fourier Theorem to write
	\[ \triangle_p(x) = p^2 \int_{(-\pi,\pi]^d} \e^{-\i k \cdot x } \ftau(k)^3 \frac{\dd k}{(2\pi)^d} \]
and then to use an assumed bound on $f_2$ to replace $\ftau$ by $\fgmu$. For site percolation, this poses a problem, since we are missing one factor of $p$. Overcoming this issue poses a novelty of Section~\ref{sec:bootstrap_analysis}. The following two observations turn out to be helpful for this:}

\begin{observation}[Convolutions of $\jeq$]\label{obs:J_convolutions}
Let $m\in\N$ and $x\in\Zd$ with $m \geq |x|$. Then there is a constant $c=c(m,x)$ with $c \leq m!$ such that
	\[ \jeq^{\ast m}(x) = c \mathds 1_{\{m-|x| \text{ is even}\}} (2d)^{(m-|x|)/2 }.  \]
\end{observation}
\begin{proof}
This is an elementary matter of counting the number of $m$-step walks from $\orig$ to $x$. If $m-|x|$ is odd, then there is no way of getting from $\orig$ to $x$ in $m$ steps.

So assume that $m-|x|$ is even. To get from $\orig$ to $x$, $|x|$ steps must be chosen to reach $x$. Only taking these $|x|$ steps (in any order) would amount to a shortest $\orig$-$x$-path. Out of the remaining steps, half can be chosen freely (each producing a factor of $2d$), and the other half must compensate them. In counting the different walks, we have to respect the at most $m!$ unique ways of ordering the steps.
\end{proof}
We remark that this also shows that the maximum is attained for $x=\orig$ when $m$ is even and for $x$ being a neighbor of $\orig$ when $m$ is odd.

\begin{observation}[Elementary bounds on $\taup^{\ast n}$] \label{obs:tau_J_extraction}
Let $n, m\in\N$. Then there is $c=c(m,n)$ such that, for all $p\in[0,1]$ and $x \in\Zd$,
	\[ \taup^{\ast n}(x) \leq c \sum_{l=0}^{m-n} p^{l} \jeq^{\ast (l+n)}(x) + c \sum_{j=1 \vee (n-m)}^{n} p^{m+j-n}(\jeq^{\ast m} \ast \taup^{\ast j})(x), \]
where we use the convention that $\sum_{l=0}^{m-n}$ vanishes for $n>m$.
\end{observation}
\begin{proof}
The observation heavily relies on the bound 
	\eqq{ \taup(x) \leq \jeq(x) + \E\Big[ \sum_{y \in \omega} \mathds 1_{\{ |y|=1, y \longleftrightarrow x \}} \Big] = \jeq (x) + p (\jeq\ast\taup)(x). \label{eq:boot:tau_bound_tau_tilde}  }
\col{Note that the left-hand side equals 0 when $x=\orig$.} We prove the statement by induction on $m-n$; for the base case, let $m\leq n$. Then we apply~\eqref{eq:boot:tau_bound_tau_tilde} to $m$ of the $n$ convoluted $\taup$ terms to obtain
	\al{\taup^{\ast n}(x) &\leq \big(\taup^{\ast(n-m)} \ast \big(\jeq + p(\jeq\ast\taup) \big)^{\ast m}  \big)(x) \\
				& = \sum_{l=0}^{m} \binom{m}{l} p^l \big( \jeq^{\ast m} \ast \taup^{\ast (n-m+l)} \big)(x) 
								= \sum_{l=n-m}^{n} \binom{m}{l+m-n} p^{l+m-n} \big( \jeq^{\ast m} \ast \taup^{\ast l} \big)(x). }
Let now $m-n>0$. Applying~\eqref{eq:boot:tau_bound_tau_tilde} once yields a sum of two terms, namely
	\eqq{ \taup^{\ast n}(x) \leq \big(\jeq \ast \taup^{\ast (n-1)}\big)(x) + p \big(\jeq\ast\taup^{\ast n}\big)(x). \label{eq:boot:tau_bound_convol_base}}
We can apply the induction hypothesis on the second term with $\tilde m = m-1$ and $\tilde n=n$, producing terms of the sought-after form. Now, observe that application of~\eqref{eq:boot:tau_bound_tau_tilde} yields
	\eqq{ \big(\jeq^{\ast j} \ast \taup^{\ast (n-j)}\big)(x) \leq \big( \jeq^{\ast(j+1)} \ast \taup^{\ast(n-j-1)}\big) (x) + p \big( \jeq^{\ast(j+1)} \ast \taup^{\ast(n-j)}\big)(x)  \label{eq:boot:tau_bound_convol_step}}
for $1 \leq j <n$. For every $j$, the second term can be bounded by the induction hypothesis for $\tilde m=m-j-1$ and $\tilde n = n-j$ (so that $\tilde m - \tilde n < m-n$) with suitable $c(m,n)$. Hence, we can iteratively break down~\eqref{eq:boot:tau_bound_convol_base}; after applying~\eqref{eq:boot:tau_bound_convol_step} for $j=n-1$, we are left with the term $\jeq^{\ast n}(x)$, finishing the proof.
\end{proof}

We now define
	\[V_p^{(m,n)}(a) := (\jeq^{\ast m} \ast \taup^{\ast n})(a), \quad W_p^{(m,n)}(a;k) := (\taupk \ast V_p^{(m,n)})(a),  \quad \widetilde W_p^{(m,n)}(a;k) := (\jek \ast V_p^{(m,n)})(a).\]
Note that $W_p$ from Definition~\ref{def:displacement_quantities} relates to the above definition via $W_p = pW_p^{(0,0)} + pW_p^{(0,1)}$. \col{Moreover, $\triangle_p(x) = p^2 V^{(0,3)}(x) $.}

\begin{lemma}[Bounds on $V_p^{(m,n)}, W_p^{(m,n)}, \widetilde W_p^{(m,n)}$] \label{lem:bootstrap:bounds_on_V}
Let $p \in [0,p_c)$ and $m,n\in\N_0$ with $d>\tfrac{20}{9} n$. Then there is a constant $c_f=c(m,n,f(p))$ (increasing in $f$) such that the following hold true:
\begin{enumerate}
\item For $m+n \geq 2$,
	\[  p^{m+n-1} V_p^{(m,n)}(a) \leq \begin{cases} c_f &\mbox{if } m+n =2 \text{ and } a=\orig, \\ c_f/d & \mbox{else}.	\end{cases}  \]
\item For $m+n \geq 1$, and under the additional assumption $d >2n+4$ for the bound on $W_p^{(m,n)}$,
	\[ p^{m+n}  \max \Big\{ \sup_{a \in \Zd}  \widetilde W_p^{(m,n)}(a;k), \sup_{a\in\Zd} W_p^{(m,n)}(a;k) \Big\}
					\leq [1-\fconnf(k)] \times \begin{cases} c_f &\mbox{if } m+n \leq 2, \\ c_f/d & \mbox{if } m+n \geq 3.	\end{cases}  \]
\end{enumerate}
\end{lemma}

\col{We apply Lemma~\ref{lem:bootstrap:bounds_on_V} for $n \leq 3$, and so $d\geq 7 > 60/9$ for the dimension suffices.}

\begin{proof}
\underline{Bound on $V_p$.} 
We start with the case $m\ge4$ where we can rewrite the left-hand side via Fourier transform and apply H\"older's inequality to obtain
	\begin{align}p^{m+j-1} (\jeq^{\ast m} \ast\taup^{\ast j})(a) &= p^{m+j-1} \int_{\fspace} \e^{-\i k \cdot a} \fjeq(k)^m \ftau(k)^j \frac{\dd k}{(2\pi)^d} \notag\\
		& \leq \bigg( p^{10(m-1)} \int_{\fspace} \fjeq(k)^{10m} \frac{\dd k}{(2\pi)^d} \bigg)^{1/10} \bigg(\int_{\fspace} (p |\ftau(k)|)^{10j/9} \frac{\dd k}{(2\pi)^d} \bigg)^{9/10} \notag\\
		& \leq  \Big( p^{10(m-1)} \jeq^{\ast 10m}(\orig) \Big)^{1/10} \times f_2(p)^j \bigg(\int_{\fspace} \fgmu(k)^{10j/9} \frac{\dd k}{(2\pi)^d} \bigg)^{9/10}. 
		\label{eq:boot:bounds_on_V_Hoelder}	\end{align}
We note that the number $10$ in the exponent holds no special meaning other than that it is large enough to make the following arguments work. The first factor in~\eqref{eq:boot:bounds_on_V_Hoelder} is handled by Observation~\ref{obs:J_convolutions}, as
	\al{ \Big( p^{10(m-1)} \jeq^{\ast 10m}(\orig) \Big)^{1/10} &\leq \Big( c p^{10(m-1)} (2d)^{5m} \Big)^{1/10} \leq \Big( c (2dp)^{10(m-1)} (2d)^{-5m+10} \Big)^{1/10} \\
			&\leq c f_1(p)^{m-1} (2d)^{-m/2+1} \leq c_f /d }
and $m \geq 4$. 
Regarding the second factor in~\eqref{eq:boot:bounds_on_V_Hoelder}, note that $10j/9 \leq 10n/9 < d/2 $ and so Proposition~\ref{thm:random_walk_triangle} gives a  uniform upper bound. 
\col{We remark that the exponent $10$ in~\eqref{eq:boot:bounds_on_V_Hoelder} is convenient because it is even and allows us to apply Lemma~\ref{lem:bootstrap:bounds_on_V} in dimension $d\ge7$.}

If $m<4$, then we first use Observation~\ref{obs:tau_J_extraction} with $\tilde m= 4-m$ to get that $p^{m+n-1}V_p^{(m,n)}$ is bounded by a sum of terms of two types, which are constant multiples of
	\eqq{ p^{l+m+n-1} \jeq^{\ast (l+m+n)}(a) = p^{s-1} \jeq^{\ast s}(a) \qquad \text{and} \quad p^{4+j-1} (\jeq^{\ast 4} \ast\taup^{\ast j})(a), \label{eq:boot:bounds_on_V_types}}
where $0 \leq l \leq 4-m-1$ (and therefore $s \geq 2$) and $1 \leq j \leq n$. 
If $s$ is odd, we can write $s= 2r+1$ for some $r \geq 1$, and Observation~\ref{obs:J_convolutions} gives
	\[ p^{2r} \jeq^{\ast (2r+1)}(a) \leq c p^{2r} (2d)^{r} = c (2dp)^{2r} (2d)^{-r} \leq c (f_1(p))^{2r} (2d)^{-r} \leq c_f/d.\]
Similarly, if $s$ is even and $a \neq \orig$ or $s \geq 4$, then $p^{s-1} \jeq^{\ast s}(a) \leq c_f/d$. Finally, if $s=2$ and $a=\orig$, then $p (\jeq\ast\jeq)(\orig) \leq c_f$.
This shows that the terms of the first type in~\eqref{eq:boot:bounds_on_V_types} are of the correct order. 
The second type is of the form $p^{4+j-1} V_p^{(4,j)}(a)$ and included in the previous considerations. 
Together this proves the claimed bound on $V_p^{(m,n)}$.

\underline{Bound on $\widetilde W_p$.} Let first $m+n \geq 3$. Then
	\al{ p^{m+n} \widetilde W_p^{(m,n)}(a;k) &= p^{m+n} \sum_{y\in\Zd} \jek(y) V_p^{(m,n)}(a-y) \\
		& \leq p^{m+n-1} \Big( \sup_{a\in\Zd} V_p^{(m,n)}(a) \Big) (2dp) \sum_{y\in\Zd} [1-\cos(k\cdot y)] \connf(y)\\
		& \leq c_f/d \times f_1(p) [1-\fconnf(k)], }
applying the bound on $V_p$. 

Consider now $m+n=2$. Using first that $\jeq \leq \taup$ and then~\eqref{eq:boot:tau_bound_tau_tilde},
	\eqq{ p^2 \widetilde W_p^{(m,n)}(a;k) \leq p^2 \widetilde W_p^{(0,2)}(a;k) \leq p^2 \widetilde W_p^{(2,0)}(a;k) + p^3 \widetilde W_p^{(2,1)}(a;k) + p^3 \widetilde W_p^{(1,2)}(a;k).
					\label{eq:boot:V_1_bound} }
The second and third summand right-hand side of~\eqref{eq:boot:V_1_bound} can be dealt with as before, we only have to deal with the first summand. Indeed, 
	\[ p^2 \widetilde W_p^{(2,0)}(a;k) = p^2 \sum_{y} \jek(y) \jeq^{\ast 2}(a-y) \leq 2d p^2 \jeq^{\ast 2}(\orig) \sum_{y} D_k(y) = f_1(p)^2 [1-\fconnf(k)], \]
and we can choose $c_f = f_1(p)^2$. 

Finally, for $m+n=1$, we have $p \widetilde W_p^{(m,n)}(a;k) \leq p \widetilde W_p^{(1,0)}(a;k) + p^2 \widetilde W_p^{(1,1)}(a;k)$. The second term was already bounded, the first is
	\[ p(\jek\ast\jeq)(a) \leq p\sum_y \jek(y) = f_1(p) [1-\fconnf(k)].\]

\underline{Bound on $W_p$.} We note that a combination of~\eqref{eq:boot:tau_bound_tau_tilde} and the Cosine-split lemma~\ref{lem:cosinesplitlemma} yields
	\eqq{ \taupk(x) \leq \jek(x) + 2p (\jek\ast\taup)(x) + 2p (\jeq\ast\taupk)(x). \label{eq:boot:taupk_jek_split_bound}}
Applying this repeatedly, we can bound $p^{m+n} W_p^{(m,n)}(a;k)$ by a sum of quantities of the form $ p^{s+t} \widetilde W_p^{(s,t)}$ (where $s+t \geq 1$) plus $c(m,n) p^{m+n} W_p^{(m,n)}$, where we can now assume $m \geq 4$ w.l.o.g. The terms of the form $\widetilde W_p^{(s,t)}$ were bounded above already. Similarly to how we obtained the bound~\eqref{eq:boot:bounds_on_V_Hoelder}, we bound the last term by applying H\"older's inequality, and so
	\algn{p^{m+n} W_p^{(m,n)}& (a;k) \leq p^{m+n} \int_{\fspace} |\fjeq(l)|^m |\ftau(l)|^n |\ftaupk(l)| \frac{\dd l}{(2\pi)^d} \notag\\
		& \leq \bigg( p^{10(m-1)} \int_{\fspace} \fjeq(l)^{10m} \frac{\dd l}{(2\pi)^d} \bigg)^{1/10} \bigg(\int_{\fspace} (p |\ftau(l)|)^{10n/9} (p |\ftaupk(l)|)^{10/9} \frac{\dd l}{(2\pi)^d} \bigg)^{9/10} \notag\\
		& \leq c_f /d \times 3000 f(p)^{n+1} \bigg(\int_{\fspace} \fgmu(l)^{10n/9} \Big[ \fgmu(l) \big(\fgmu(l-k) + \fgmu(l+k)\big) \notag\\
		& \hspace{5cm} + \fgmu(l-k)\fgmu(l+k)\Big]^{10/9}  \frac{\dd l}{(2\pi)^d} \bigg)^{9/10} \notag\\
		& \leq c_f/d, \label{eq:boot:V_2_bound} }
where the last bound is due to Proposition~\ref{thm:random_walk_triangle_related} and the value of $c_f$ has changed in the last line.
\end{proof}

The proofs of the following lemmas, bounding the quantities appearing in Section~\ref{sec:diag_bounds}, are direct consequences of Lemma~\ref{lem:bootstrap:bounds_on_V}.

\begin{lemma}[Bounds on various triangles] \label{lem:bootstrap:bounds_on_trip}
Let $p \in [0,p_c)$ and $d>6$. Then there is $c_f=c(f(p))$ (increasing in $f$) such that
	\[ \max\{ \trip,\tripo, \tripf, \tripof\} \leq c_f /d, \qquad \max\{\tripf(\orig), \tripof(\orig),  \tripoff\} \leq c_f.  \]
\end{lemma}
\begin{proof}
Note that
	\al{ \trip(x) = p^2 V_p^{(0,3)}(x),& \qquad  \tripo(x) = p^2V_p^{(0,2)}(x) + \trip(x), \qquad \tripf(x) = p V_p^{(0,2)}(x) + \trip(x), \\
			 \tripof(x) &= p\taup(x) + \tripf(x), \qquad \tripoff(x) = \delta_{\orig,x} + \tripof(x).}
For the bound on $p \taup \leq p$, we use that $p \leq f_1(p)/d$. For all remaining quantities, we use Lemma~\ref{lem:bootstrap:bounds_on_V}, which is applicable since $n \leq 3$ and $\tfrac{20}{9} n \leq \tfrac{60}{9} < 7 \leq d$.
\end{proof}

\begin{lemma}[Bound on $W_p$] \label{lem:bootstrap:bounds_on_W}
Let $p \in [0,p_c)$ and $d>6$. Then there is a constant $c_f=c(f(p))$ (increasing in $f$) such that
	\[ W_p(k) \leq [1-\fconnf(k)] c_f.  \]
\end{lemma}
\begin{proof}
By~\eqref{eq:boot:taupk_jek_split_bound}, 
	\al{W_p(x;k) & = p W_p^{(0,1)}(x;k) + p \taupk(x) \\
		& \leq p W_p^{(0,1)}(x;k) + 2p^2 \widetilde W_p^{(0,1)}(x;k) + 2p^2 W_p^{(1,0)}(x;k) + p \jek(x).}
The proof follows from Lemma~\ref{lem:bootstrap:bounds_on_V} together with the observation that
	\[ p\jek(x) = (2dp) \connf_k(x) \leq f_1(p) \sum_{x\in\Zd} \connf_k(x) = f_1(p) [1-\fconnf(k)]. \qedhere\]
\end{proof}

\begin{lemma}[Bounds on $\Pi_p^{(0)}$ and $\Pi_p^{(1)}$] \label{lem:bootstrap:bounds_on_Pi0}
Let $p \in [0,p_c), i\in\{0,1\}$, and $d>6$. Then there is a constant $c_f=c(f(p))$ (increasing in $f$) such that
	\[ p \sum_x \Pi_p^{(0)}(x) \leq c_f/d, \qquad p \sum_x [1-\cos(k\cdot x)] \Pi_p^{(i)}(x) \leq [1-\fconnf(k)] c_f/d.  \]
\end{lemma}
\begin{proof}
We recall the two bounds obtained in Proposition~\ref{thm:db:bounds_for_n0}. The first one yields $p|\widehat\Pi_p^{(0)}(k)| \leq p^3 V_p^{(2,2)}(\orig)$, the second one yields
	\[ p \widehat\Pi_p^{(0)}(0) - p\widehat\Pi_p^{(0)}(k) \leq 2p^3 \widetilde W_p^{(1,2)}(\orig;k) + 2p^3 W_p^{(2,1)}(\orig;k).\]
All of these bounds are handled directly by Lemma~\ref{lem:bootstrap:bounds_on_V}. Similarly, the only quantity in the bound of Proposition~\ref{thm:db:displacement_thm_n1} that was not bounded already is $p^2 W_p^{(1,1)}(\orig;k)$. By a combination of~\eqref{eq:boot:tau_bound_tau_tilde} and~\eqref{eq:boot:taupk_jek_split_bound}, we can bound
	\al{ p^2 W_p^{(1,1)}(\orig;k) & \leq p^2 \Big( W_p^{(2,0)}(\orig;k) + pW_p^{(2,1)}(\orig;k) \Big) \\
		& \leq p^2 \Big( \widetilde W_p^{(2,0)}(\orig;k) + 2p \widetilde W_p^{(2,1)}(\orig;k) + 2p W_p^{(3,0)}(\orig;k) + p W_p^{(2,1)}(\orig;k) \Big).}
But $0 \leq \widetilde W_p^{(2,0)}(\orig;k) \leq 2 \jeq^{\star 3}(\orig) = 0$ by Observation~\ref{obs:J_convolutions}. The other three terms are bounded by Lemma~\ref{lem:bootstrap:bounds_on_V}.
\end{proof}

\begin{lemma}[Displacement bounds on $H_p$] \label{lem:bootstrap:bounds_on_H}
Let $p \in [0,p_c)$ and $d>6$. Then there is a constant $c_f=c(f(p))$ (increasing in $f$) such that
	\[ H_p(k) \leq [1-\fconnf(k)] c_f/d.  \]
\end{lemma}
\begin{proof} We recall that
	\[H_p(b_1, b_2;k) = p^5 \hspace{-0.15cm} \sum_{t,w,z,u,v} \taup(z) \taup(t-u) \taup(t-z) \taupk(u-z) \taup(t-w) \taup(w-b_1) \taup(v-w) \taup(v+b_2-u).\]
We bound the factor $\taup(z) \leq \jeq(z) + p(\jeq\ast\taup)(z)$, splitting $H_p$ into a sum of two. The first term is easy to bound. Indeed,
	\al{p^5 \sum_{t,w,z} & \jeq(z) \taup(t-z) \taup(w-t) \taup(b_1-w) \sum_u \taupk(u-z)\taup(t-u) (\taup\ast\taup)(b_2-u-w) \\
			& \leq \tripf(\orig)  p^4 \sum_{t,w,z,u} \jeq(z) \taup(t-z) \taup(w-t) \taup(b_1-w) (\taupk\ast\taup)(t-z) \\
			& \leq \tripf(\orig) W_p(k) p^3 V_p^{(1,3)}(b_1) \leq [1-\fconnf(k)] c_f/d }
by the previous Lemmas~\ref{lem:bootstrap:bounds_on_V}-\ref{lem:bootstrap:bounds_on_W}. We can thus focus on bounding
	\eqq{ p^6 \hspace{-0.15cm} \sum_{t,w,z,u,v} (\jeq\ast\taup)(z) \taup(t-u) \taup(t-z) \taupk(u-z) \taup(t-w) \taup(w-b_1) \taup(v-w) \taup(v+b_2-u). \label{eq:boot:H_bounds_conv_z} }
To prove such a bound (and thus the lemma), we need to recycle some ideas from the proof of Lemma~\ref{lem:bootstrap:bounds_on_V} in a more involved fashion. To this end, let
	\[ \sigma(x) := p^4(\jeq^{\ast 4}\ast\taup)(x) + \sum_{j=1}^{4} p^{j-1} \jeq^{\ast j}(x),\]
and note that $\taup(x) \leq \sigma(x)$ by~\eqref{eq:boot:tau_bound_tau_tilde}. Consequently,~\eqref{eq:boot:H_bounds_conv_z} is bounded by $\widetilde H_p(a_1,a_2;k)$, where we define 
	\[\widetilde H_p(a_1,a_2;k) = p^6 \sum_{t,w,z,u,v} (\jeq\ast\sigma)(z)\sigma(t-u)\sigma(t-z)\taupk(u-z)\sigma(t-w)\sigma(w-a_1) \sigma(v-w)\sigma(v+a_2-u). \]
By the Inverse Fourier Theorem, we can write
	\al{ \widetilde H_p(a_1,a_2;k) = p^6 \int_{(-\pi,\pi]^{3d}} \hspace{-0.3cm} \e^{-\i a_1 \cdot l_1-\i a_2 \cdot l_2} & \; \fjeq(l_1)\widehat\sigma(l_1)^2 \widehat\sigma(l_2)^2\ftaupk(l_3)
					 \widehat\sigma(l_1-l_2) \widehat\sigma(l_1-l_3) \widehat\sigma(l_2-l_3) \frac{\dd(l_1,l_2,l_3)}{(2\pi)^{3d}}. }
(For details on the above identity, see~\cite[Lemma 5.7]{HeyHofLasMat19} and the corresponding bounds on $H_\lambda$ therein.) We bound
	\begin{align}
		\widetilde H_p(a_1,a_2;k) &\leq 3000 f_3(p) [1-\fconnf(k)] p^5 \int_{(-\pi,\pi]^{3d}} |\fjeq(l_1)|\widehat\sigma(l_1)^2 \widehat\sigma(l_2)^2 
						|\widehat\sigma(l_1-l_2)| |\widehat\sigma(l_1-l_3)| |\widehat\sigma(l_2-l_3)| \notag \\
			& \times\Big( \fgmu(l_3)\fgmu(l_3-k) + \fgmu(l_3)\fgmu(l_3+k) + \fgmu(l_3-k)\fgmu(l_3+k) \Big) \frac{\dd(l_1,l_2,l_3)}{(2\pi)^{3d}}.  \label{eq:boot:H_bounds_applying_f_line1}
	\end{align}
Opening the brackets in~\eqref{eq:boot:H_bounds_applying_f_line1} gives rise to three summands. We show how to treat the third one. Applying Cauchy-Schwarz, we obtain
	\begin{align}
		 \bigg( & \int_{(-\pi,\pi]^{3d}} \big[ p^2 |\fjeq(l_1)| |\widehat\sigma(l_1)|^3\big]\big[p^2 \widehat\sigma(l_2-l_1)^2 |\widehat\sigma(l_2)|\big]
		 		 	\big[p \fgmu(l_3+k)^2 |\widehat\sigma(l_3-l_2)| \big]  \frac{\dd(l_1,l_2,l_3)}{(2\pi)^{3d}} \bigg)^{1/2} 	\label{eq:boot:H_bounds_CS_1}\\
			& \times \bigg( \int_{(-\pi,\pi]^{3d}} \big[p^2 |\widehat\sigma(l_2)|^3 \big] \big[ p^2\widehat\sigma(l_1-l_3)^2 |\fjeq(l_1)| |\widehat\sigma(l_1)| \big]
					 \big[p \fgmu(l_3-k)^2 |\widehat\sigma(l_3-l_2)| \big] \frac{\dd(l_1,l_2,l_3)}{(2\pi)^{3d}} \bigg)^{1/2}	\label{eq:boot:H_bounds_CS_2}
	\end{align}	
The square brackets indicate how we want to decompose the integrals. We first bound~\eqref{eq:boot:H_bounds_CS_1}, and we start with the integral over $l_3$. We intend to treat the five summands constituting $\widehat\sigma(l_3-l_2)$ simultaneously. Indeed, note that with our bound on $f_2$,
	\eqq{ |\widehat\sigma(l)| \leq \sum_{j=1}^{4} p^{j-1} |\fjeq(l)|^j + p^3\fjeq(l)^4 \fgmu(l) \leq 5 \max_{n\in\{0,1\}, j\in[4]} p^{(j \vee 4n) -1} |\fjeq(l)|^{(j \vee 4n)} \fgmu(l)^n.
					\label{eq:boot:gamma_fourier_bound}  }
With this,
	\al{ p^{(j \vee 4n)} & \int_{\fspace} |\fjeq(l_3-l_2)|^{(j \vee 4n)} \fgmu(l_3-l_2)^n \fgmu(l_3+k)^2 \frac{\dd l_3}{(2\pi)^d} \\
		\leq & \Big( p^{10 (j \vee 4n)} \int_{\fspace} \fjeq(l_3)^{10(j \vee 4n)} \frac{\dd l_3}{(2\pi)^d} \Big)^{1/10} \\		
			& \hspace{1cm} \times \Big( \int_{\fspace} \fgmu(l_3+k)^{20/9} \big[\fgmu(l_3-l_2) + \fgmu(l_3-2k + l_2)  \big]^{10/9} \frac{\dd l_3}{(2\pi)^d} \Big)^{9/10} \\
		\leq & (c_f/d)^{1/2}  }
by the same considerations that were performed in~\eqref{eq:boot:V_2_bound}. We use the same approach to treat the integral over $l_2$ in~\eqref{eq:boot:H_bounds_CS_1}. Applying~\eqref{eq:boot:gamma_fourier_bound} to all three factors of $\widehat\sigma$ gives rise to tuples $(j_i,n_i)$ for $i\in[3]$, and so
	\al{ p^{-1+\sum_{i=1}^{3} (j_i \vee 4n_i)} & \int_{\fspace} |\fjeq(l_2-l_1)|^{\sum_{i=1}^{2}(j_i\vee 4n_i)} |\fjeq(l_2)|^{j_3 \vee 4n_3} \fgmu(l_2-l_1)^{n_1+n_2} \fgmu(l_2)^{n_3}  \frac{\dd l_2}{(2\pi)^d} \\
		\leq & \Big( p^{10(-1+\sum_{i=1}^{3}(j_i \vee 4n_i))} \int_{\space} \fjeq(l_2-l_1)^{10\sum_{i=1}^{2}(j_i\vee 4n_i)} \fjeq(l_2)^{10(j_3 \vee 4n_3)} \frac{\dd l_2}{(2\pi)^d} \Big)^{1/10} \\
			 & \hspace{1cm }\times \Big( \int_{\fspace} \fgmu(l_2-l_1)^{10(n_1+n_2)/9} \big[\fgmu(l_2) + \fgmu(l_2-2 l_1)  \big]^{10n_3/9} \frac{\dd l_2}{(2\pi)^d} \Big)^{9/10} \\
		\leq & c_f \Big( p^{20(-1+\sum_{i=1}^{2}(j_i \vee 4n_i))} \int_{\space} \fjeq(l_2-l_1)^{20\sum_{i=1}^{2}(j_i\vee 4n_i)} \frac{\dd l_2}{(2\pi)^d} \Big)^{1/20} \\
			 & \hspace{1cm} \times \Big( p^{20(j_3 \vee 4n_3)} \int_{\space} \fjeq(l_2)^{20(j_3 \vee 4n_3)} \frac{\dd l_2}{(2\pi)^d} \Big)^{1/20} \\
		\leq & (c'_f/d)^{1/2}.  }
We finish by proving that the integral over $l_1$ in~\eqref{eq:boot:H_bounds_CS_1} is bounded by a constant. Indeed,
	\al{ & \hspace{2cm} p^{-1+\sum_{i=1}^{3} (j_i \vee 4n_i)} \int_{\fspace} |\fjeq(l_1)|^{1+\sum_{i=1}^{3}(j_i\vee 4n_i)} \fgmu(l_1)^{n_1+n_2+n_3} \frac{\dd l_1}{(2\pi)^d} \\
		& \leq \Big( p^{10(-1+\sum_{i=1}^{3}(j_i \vee 4n_i))} \int_{\space} \fjeq(l_1)^{10(1+\sum_{i=1}^{3}j_i\vee 4n_i)} \frac{\dd l_1}{(2\pi)^d} \Big)^{1/10} 
				 \Big( \int_{\fspace} \fgmu(l_1)^{10(n_1+n_2+n_3)/9} \frac{\dd l_1}{(2\pi)^d} \Big)^{9/10} \\
		& \leq c_f \cdot p^{-1+\sum_{i=1}^{3}(j_i \vee 4n_i)} \Big( \jeq^{\ast 10(1+\sum_{i=1}^{3}j_i\vee 4n_i)}(\orig) \Big)^{1/10} \\
		& \leq c'_f \cdot p^{-1+\sum_{i=1}^{3}(j_i \vee 4n_i)} (2d)^{\frac 12 (1+\sum_{i=1}^{3}j_i\vee 4n_i)} \leq c''_f. }
This proves that~\eqref{eq:boot:H_bounds_CS_1} is bounded by $(c_f/d)^{1/2}$. Note that~\eqref{eq:boot:H_bounds_CS_2} is very similar to~\eqref{eq:boot:H_bounds_CS_1}, and the same bounds can be applied to get a bound of $(c_f/d)^{1/2}$. Since the other two terms in~\eqref{eq:boot:H_bounds_applying_f_line1} are handled analogously, we obtain the bound $\widetilde H_p(b_1,b_2;k) \leq [1-\fconnf(k)] c_f/d$, which is what was claimed. 
\end{proof}

\begin{proof}[Proof of Proposition~\ref{thm:convergence_of_LE}]
Recalling the bounds on $|\Pi_p^{(m)}(k)|$ obtained in Propositions~\ref{thm:db:bounds_for_n0} and~\ref{thm:db:main_thm}, and the bounds on $|\Pi_p^{(m)}(k) - \Pi_p^{(m)}(0)|$ obtained in Propositions~\ref{thm:db:bounds_for_n0},~\ref{thm:db:displacement_thm_n1}, and~\ref{thm:db:displacement_thm}, we can combine them with the bounds just obtained in Lemmas~\ref{lem:bootstrap:bounds_on_trip},~\ref{lem:bootstrap:bounds_on_W},~\ref{lem:bootstrap:bounds_on_Pi0}, and~\ref{lem:bootstrap:bounds_on_H}. This gives
	\algn{ p|\widehat\Pi_p^{(m)}(k)| & \leq p \sum_{x\in\Zd} \Pi_p^{(m)}(x)  \leq c_f (c_f/d)^{m \vee 1}, \label{eq:boot:Pi_m_exp_bound}\\
		 p|\widehat\Pi_p^{(m)}(k) - \widehat\Pi_p^{(m)}(0)| &= p\sum_{x\in\Zd} [1-\cos(k\cdot x)] \Pi_p^{(m)}(x) \leq c_f [1-\fconnf(k)] (c_f/d)^{1\vee (m-2)}. \notag }
Summing the above terms over $m$, we recognize the geometric series in their bounds. The series converges for sufficiently large $d$. If $f \leq 4$ on $[0,p_c)$, we can replace $c_f$ by $c=c_4$ in the above, so that the bounds are uniform in $p\in[0,p_c)$, which means that the value of $d$ above which the series converges is independent of $p$. Hence,
	\[ p|\widehat\Pi_n(k)| \leq \sum_{m=0}^{\infty} p \Pi_p^{(m)}(x) \leq c_f/d, \qquad p|\widehat\Pi_n(k)-\widehat\Pi_n(0)| \leq [1-\cos(k\cdot x)] c_f/d.\]
The bound~\eqref{eq:boot:convergence_of_Pi_n_unsummed} follows from Corollary~\ref{cor:db:Pi_pn_unsummed_bound} by analogous arguments. Recalling the definition of the remainder term yields the straight-forward bound
	\eqq{ \sum_x R_{p,n}(x) \leq \sum_x \sum_u p \Pi_p^{(n)}(u) \taup(x-u) \leq p\widehat\Pi_p^{(n)}(0) \ftau(0) \leq c_f (c_f/d)^n \ftau(0), \label{eq:boot:R_n_bound}}
applying~\eqref{eq:boot:Pi_m_exp_bound}. Hence, if $c_f/d<1$ and $p<p_c$, then $\sum_x R_{p,n}(x) \to 0$ as $n\to\infty$. Again, if $f\leq 4$, we can replace $c_f$ by $c=c_4$ in~\eqref{eq:boot:R_n_bound} and the smallness of $(c/d)$ does not depend on the value of $p$.

The existence of the limit $\Pi_p$ follows by dominated convergence with the bound~\eqref{eq:boot:convergence_of_Pi_n_unsummed}. Together with~\eqref{eq:boot:R_n_bound}, this implies that the lace expansion identity in Proposition~\ref{thm:lace_expansion} converges as $n\to\infty$ and satisfies the OZE.
\end{proof}

Corollary~\ref{cor:boot:OZE_at_critical_point} as well as the main theorem now follow from Proposition~\ref{thm:convergence_of_LE} in conjunction with Proposition~\ref{thm:bootstrap:forbidden_region}, which is proven in Section~\ref{sec:boot:bootstrap_argument} below.

\begin{proof}[Proof of Corollary~\ref{cor:boot:OZE_at_critical_point} and Theorem~\ref{thm:main_theorem_triangle_condition}] 
Proposition \ref{thm:bootstrap:forbidden_region} implies that indeed $f(p)\le 3$ for all $p\in[0,p_c)$, and therefore, the consequences in the second part of Proposition \ref{thm:convergence_of_LE} are valid. Lemma~\ref{lem:bootstrap:bounds_on_trip} together with Fatou's lemma and pointwise convergence of $\taup(x)$ to $\tau_{p_c}(x)$ then implies the triangle condition. 

The remaining arguments are analogous to the proofs of~\cite[Corollary 6.1]{HeyHofLasMat19} and~\cite[Theorem 1.1]{HeyHofLasMat19}. 
The rough idea for the proof of Corollary~\ref{cor:boot:OZE_at_critical_point} is to use $\theta(p_c)=0$ (which follows from the triangle condition) to couple the model at $p_c$ with the model at $p<p_c$, and then show that, as $p \nearrow p_c$, the (a.s.) finite cluster of the origin is eventually the same. 
For the full argument and the proof of the infra-red bound, we refer to~\cite{HeyHofLasMat19}.
\end{proof}

\subsection{Completing the bootstrap argument} \label{sec:boot:bootstrap_argument}

It remains to prove that $f \leq 4$ on $[0,p_c)$ so that we can apply the second part of Proposition~\ref{thm:convergence_of_LE}. This is achieved by Proposition~\ref{thm:bootstrap:forbidden_region}, where three claims are made: First, $f(0)\leq 4 $, secondly, $f$ is continuous in the subcritical regime, and thirdly, $f$ does not take values in $(3,4]$ on $[0,p_c)$. This implies the desired boundedness of $f$. The following observation will be needed to prove the third part of Proposition~\ref{thm:bootstrap:forbidden_region}:
\begin{observation} \label{obs:Delta_k_bounds}
Suppose $a(x) = a(-x)$ for all $x\in\Zd$. Then
	\[ \tfrac 12 \big\vert \Delta_k \widehat a(l) \big\vert \leq \widehat{|a|}(0) - \widehat{|a|}(k)\]
for all $k,l\in\fspace$ (where $\widehat{|a|}$ denotes the Fourier transform of $|a|$). As a consequence, $|\fconnf_k(l)| \leq 1-\fconnf(k)$. Moreover, there is $d_0 \geq 6$ a constant $c_f=c(f(p))$ (increasing in $f$) such that, for all $d > d_0$,
	\[ \big\vert \Delta_k p\fpip(l) \big\vert \leq [1-\fconnf(k)] c_f /d. \]
\end{observation}
\begin{proof}
The statement for general $a$ can be found, for example, in~\cite[(8.2.29)]{HeyHof17}. For convenience, we give the proof. Setting $a_k(x)=[1-\cos(k\cdot x)] a(x)$, we have
	\al{\tfrac 12 | \Delta_k \widehat a(l) | &= | \widehat a_k(l)| \leq \sum_{x \in \Zd} \big\vert a(x) \cos(l\cdot x) [1-\cos(k\cdot x)] \big\vert
				 \leq \sum_{x \in\Zd} |a(x)| [1-\cos(k \cdot x)] \\
		&\leq \widehat{|a|}(0) - \widehat{|a|}(k) }
The consequence about $\fconnf$ now follows from $\connf(x) \geq 0$ for all $x$. Moreover, the statement for $\Delta_k p\fpip(l)$ follows applying the observation to $a=\Pi_p$ together with the bounds in~\eqref{eq:boot:convergence_of_Pi_M}.
\end{proof}

\begin{prop} \label{thm:bootstrap:forbidden_region}
The following are true:
\begin{enumerate}
\item The function $f$ satisfies $f(0) \leq 3$.
\item The function $f$ is continuous on $[0,p_c)$.
\item Let $d$ be sufficiently large; then assuming $f(p) \leq 4$ implies $f(p) \leq 3$ for all $p \in [0,p_c)$.
\end{enumerate}
Consequently, there is some $d_0$ such that $f(p) \leq 3$ uniformly for all $p \in [0,p_c)$ and $d > d_0$.
\end{prop}

As a remark, in the third step of Proposition~\ref{thm:bootstrap:forbidden_region}, we prove the stronger statement $f_i(p) \leq 1+ \textup{const} / d$ for $i \in \{1,2\}$. 
In the remainder of the paper, we prove this proposition and thereby finish the proof the main theorem. We prove each of the three assertions separately. 

\paragraph{1. Bounds on $f(0)$.} This one is straightforward. As $f_1(p) = 2dp$, we have $f_1(0)=0$.

Further, recall that $\lambda_p =1-1/\chi(p)$, and so $\lambda_0=0$ and $\widehat{G}_{\lambda_0}(k)=\widehat{G}_0(k) = 1$ for all $k\in\fspace$. Since both $p|\ftau(k)|$ and $p |\widehat \tau_{p,k}(l)|$ equal $0$ at $p=0$, recalling the definitions of $f_2$ and $f_3$ in~\eqref{eq:boot:intro:f_functions_def}, we conclude that $f_2(0) = f_3(0) =0$. In summary, $f(0)=0$.

\paragraph{2. Continuity of $f$.} The continuity of $f_1$ is obvious. For the continuity of $f_2$ and $f_3$, we proceed as in~\cite{HeyHof17}, that is, we prove continuity on $[0,p_c-\varepsilon]$ for every $0<\varepsilon<p_c$. This again is done by taking derivatives and bounding them uniformly in $k$ and in $p \in[0,p_c-\varepsilon]$. To this end, we calculate
	\eqq{ \frac{\dd}{\dd p} \frac{\ftau(k)}{\fgmu(k)} = \frac{1}{\fgmu(k)^2} \Big[\fgmu(k) \frac{\dd\ftau(k)}{\dd p}
			 - \ftau(k) \frac{\dd \fgreenslam(k)}{\dd \lambda} \bigg\vert_{\lambda=\lambda_p} \frac{\dd \lambda_p}{\dd p} \Big]. \label{eq:boot:continuity_f2}}
Since $\lambda_p = 1-1/\chi(p)$,
	\eqq{ \frac 12 \leq \frac{1}{1-\lambda_p \fconnf(k)} = \fgmu(k) \leq \fgmu(0) = \chi(p) \leq \chi(p_c-\varepsilon). \label{eq:boot:continuity_f2_fgmu_bd} }
Further, since $\ftau(0)$ is non-decreasing,
	\eqq{ \ftau(k) \leq \ftau(0) \leq \widehat \tau_{p_c-\varepsilon}(0) = \frac{\chi(p_c-\varepsilon)-1}{p_c-\varepsilon}. \label{eq:boot:continuity_f2_fttaup_bd} }
We use Observation~\ref{obs:tools:diff_inequality} to obtain
	\[ \left\vert \frac{\dd}{\dd p} \ftau(k) \right\vert =  \bigg\vert \sum_{x\in\Zd} \e^{\i k \cdot x} \frac{\dd}{\dd p} \taup(k) \bigg\vert
		 \leq \sum_{x\in\Zd} \frac{\dd}{\dd p} \taup(x) = \frac{\dd}{\dd p} \sum_{x\in\Zd} \taup(x) \leq \ftau(0)^2, \]
and the same bound as in~\eqref{eq:boot:continuity_f2_fttaup_bd} applies. The interchange of sum and derivative is justified as both sums are absolutely summable.
Note that $\dd \fgreenslam(k)/\dd \lambda = \fconnf(k) \fgreenslam(k)^2$, and this is bounded by $\chi(p_c-\varepsilon)^2$ for $\lambda=\lambda_p$ by~\eqref{eq:boot:continuity_f2_fgmu_bd}. Finally, by Observation~\ref{obs:tools:diff_inequality},
	\[ \frac{\dd \lambda_p}{\dd p} = \frac{\frac{\dd}{\dd p}\chi(p)}{\chi(p)^2} \leq \ftau(0),\]
which is bounded by~\eqref{eq:boot:continuity_f2_fttaup_bd} again. In conclusion, all terms in~\eqref{eq:boot:continuity_f2} are bounded uniformly in $k$ and $p$, which proves the continuity of $f_2$. We can treat $f_3$ in the exact same manner, as it is composed of terms of the same type as the ones we just bounded.

\paragraph{3. The forbidden region $(3,4]$.} Note that we assume $f(p) \leq 4$ in the following, and so the second part of Proposition~\ref{thm:convergence_of_LE} applies with $c=c_4$.

\underline{Improvement of $f_1$.} Recalling the definition of $\lambda_p \in[0,1]$, this implies
	\[f_1(p) = \lambda_p - p\fpip(0) \leq 1 + c_4/d.\]

\underline{Improvement of $f_2$.} We introduce $a = p(\jeq + \Pi_p)$, and moreover
	\[ \hnk= \frac{\hak}{1+p\fpip(0)}, \qquad \hfk = \frac{1-\hak}{1+p\fpip(0)}.\]
By adapting the analogous argument from~\cite[proof of Theorem 1.1]{HeyHofLasMat19}, we can show that $1-\hak>0$. Therefore, under the assumption that $f(p) \leq 4$, we have $p\ftau(k) = \hnk/\hfk = \hak /(1-\hak)$, and furthermore $\lambda_p= \hao$. In the following lines, $M$ and $M'$ denote constants (typically multiples of $c_4$) whose value may change from line to line. An important observation is that
	\[ \frac{1}{1+p\fpip(0)} \leq 1 + M/d, \qquad |\hak| \leq 1 + M/d, \qquad |p\fpip(k)| \leq M/d. \]
We are now ready to treat $f_2$. Since $|\hnk| \leq 1+M/d$ and by the triangle inequality,
	\algn{ \bigg\vert\frac{p\ftau(k)}{\fgmu(k)} \bigg\vert &=  \Big\vert \hnk + p \ftau(k) [1- \lambda_p \fconnf(k) - \hfk] \Big\vert \notag \\
				 &\leq 1+M/d + \big\vert p \ftau(k) \big\vert \Big\vert 1- \lambda_p \fconnf(k) - \hfk \Big\vert. \label{eq:boot:FR:f2_main_identity}}
Also,
	\algn{\big\vert 1- \lambda_p \fconnf(k) - \hfk \big\vert &= 
						\bigg\vert\frac{1+p\fpip(0)-\big(2dp + p\fpip(0)\big)\big(1+p\fpip(0)\big)\fconnf(k) - 1 + 2dp \fconnf(k) + p\fpip(k)}{1+p\fpip(0)} \bigg\vert \notag\\
		&= \bigg\vert\frac{[1-\fconnf(k)] p\fpip(0)}{1+p\fpip(0)} \bigg\vert + \bigg\vert\frac{\hao p\fpip(0) \fconnf(k) + p\fpip(k)}{1+p\fpip(0)} \bigg\vert. \label{eq:boot:FR:f2_decomp_step}}
The first term in the right-hand side of~\eqref{eq:boot:FR:f2_decomp_step} is bounded by $[1-\fconnf(k)] M/d$. In the second term, recalling that $\hao=\lambda_p$, we add and subtract $p\fpip(0)$ in the numerator, and so 
	\al{ \big\vert 1- \lambda_p \fconnf(k) - \hfk \big\vert &\leq  [1-\fconnf(k)] M/d + \bigg\vert \frac{[1-\lambda_p \fconnf(k)] p\fpip(0) + p\big(\fpip(k)-\fpip(0)\big)}{1+p\fpip(0)} \bigg\vert \\
		& \leq [1-\fconnf(k)] M'/d + [1-\lambda_p \fconnf(k)] M'/d \leq 3 [1-\lambda_p \fconnf(k)] M'/d  }
for some constant $M'$. In the second to last bound, we have used that $p\fpip(0) - p\fpip(k) \leq [1-\fconnf(k)] M/d$. For the last, we have used that $1-\fconnf(k) \leq 2(1-\lambda_p\fconnf(k)) = 2 \fgmu(k)^{-1}$. Putting this back into~\eqref{eq:boot:FR:f2_main_identity}, we obtain
\[ \left\vert \frac{p \ftau(k)}{\fgmu(k)} \right\vert \leq 1+ M/d + 3|p \ftau(k) / \fgmu(k)| M'/d \leq 1+4(M \vee M') /d.\]
This concludes the improvement on $f_2$. Before dealing with $f_3$, we make an important observation:
\begin{observation}\label{obs:bootstrap:FR:f4_bound}
Given the improved bounds on $f_1$ and $f_2$,
	\[ \sup_{k\in\fspace} \left\vert \frac{1-\lambda_p \fconnf(k)}{1-\hak} \right\vert \leq 3. \]
\end{observation}
\begin{proof}
Consider first those $p$ such that $2dp \leq 3/7$. Then
	\[ \left\vert \frac{1-\lambda_p \fconnf(k)}{1-\hak} \right\vert = \left\vert \frac{1- 2dp \fconnf(k) - p \fpip(0) \fconnf(k)}{1- 2dp \fconnf(k) - p\fpip(k)} \right\vert
					\leq \frac{\frac{10}{7} + M/d}{\frac 47 + M/d} \leq (1+ M' /d) \tfrac 52 \leq 3 \]
for $d$ sufficiently large. Next, consider those $k\in\fspace$ such that $|\fconnf(k)| \leq 7/8$. Then
	\[ \left\vert \frac{1-\lambda_p \fconnf(k)}{1-\hak} \right\vert = \left\vert 1 - \frac{p\fpip(0) \fconnf(k) - p\fpip(k)}{1-\hak} \right\vert
					 \leq 1 + \frac{2M /d}{1-(1+M/d) \frac 78 - M/d}  \leq 1 + 16 M' /d \]
for $d$ sufficiently large. Let now $p$ such that $2dp> 3/7$ and $k$ such that $|\fconnf(k)| > 7/8$. We write $1-\lambda_p \fconnf(k) = \fgmu(k)^{-1}$ and $1-\hak = \ftau(k)/\hak$. Since $2dp|\fconnf(k)| \geq 3/8$, we obtain
	\al{ \left\vert \frac{1-\lambda_p \fconnf(k)}{1-\hak} \right\vert & \leq \frac 83 \left\vert \frac{\ftau(k)}{\fgmu(k)} \cdot \frac{2dp \fconnf(k)}{\hak} \right\vert
					\leq \frac 83 \big(1+M /d\big) \left\vert 1 - \frac{p\fpip(k)}{\hak} \right\vert \\
		& \leq \frac 83 \big(1+M /d\big) \Big(1+ \frac{M/d}{\frac 83 - M /d}   \Big) \leq 3  }
for $d$ sufficiently large.
\end{proof}
\underline{Improvement of $f_3$.} Elementary calculations give
	\[ \Delta_k p \ftau(l) = \underbrace{\frac{\Delta_k \widehat a(l)}{1-\widehat a(l)}}_{\text{(I)}}
						+ \sum_{\sigma\in \pm 1} \underbrace{\frac{\big(\widehat a(l+\sigma k) - \widehat a(l) \big)^2}{(1-\widehat a(l))(1-\widehat a(l+\sigma k))}}_{\text{(II)}}
						+ \underbrace{\widehat a(l) \Delta_k \Big( \frac{1}{1-\widehat a(l)} \Big)}_{\text{(III)}}. \]
We bound each of the three terms (I)-(III) separately. For the first term,
	\al{|\text{(I)}| & = \big\vert \Delta_k \widehat a(l) \big\vert \cdot \left\vert \frac{1-\lambda_p \fconnf(l)}{1-\widehat a(l)} \right\vert \cdot \fgmu(l)
					\leq 3 \fgmu(l) \Big\vert 2dp \Delta_k \fconnf(l) + \Delta_k p \fpip(l) \Big\vert \\
					& \leq (3 + M/d) [1-\fconnf(k)] \fgmu(l) \fgmu(l+k).  }
In the above, we have first used Observation~\ref{obs:bootstrap:FR:f4_bound}, then Observation~\ref{obs:Delta_k_bounds}, and finally the fact that $\fgmu(l+k) \geq 1/2$. Note that if we obtained similar bounds on (II) and (III), we could prove a bound of the form $|\Delta_k \ftau(l)| \leq c \ulam(k,l)$ for the right constant $c$ and the improvement of $f_3$ would be complete.

To deal with (II), we need a bound on $\widehat a(l+\sigma k)-\widehat a(l)$ for $\sigma\in\{\pm 1\}$. As in~\cite[(8.4.19)-(8.4.21)]{HeyHof17},
	\al{|\fconnf(l\pm k) - \fconnf(l)| & \leq \sum_x \Big( | \sin(k \cdot x)| \connf(x) + [1-\cos(k \cdot x)] \connf(x) \Big) = 1-\fconnf(k) + \sum_x | \sin(k \cdot x)| \connf(x) \\
		& \leq 1-\fconnf(k) + \Big( \sum_x \connf(x) \Big)^{1/2} \Big( \sum \sin(k \cdot x)^2 D(x) \Big)^{1/2} \\
		& \leq 1-\fconnf(k) + 2 \Big( \sum [1-\cos(k\cdot x)] D(x) \Big)^{1/2} \\
		& \leq 4 [1-\fconnf(k)]^{1/2},}
and similarly
	\al{p|\fpip(l \pm k) - \fpip(l)| & \leq \Big(p \sum_x |\Pi_p(x)| \Big)^{1/2} \Big( 2p \sum_x [1-\cos(k\cdot x)] |\Pi_p(x)| \Big)^{1/2} + p\sum_x [1-\cos(k\cdot x)] |\Pi_p(x)| \\
		& \leq  M [1-\fconnf(k)]^{1/2} /d.}
Putting this together yields
	\eqq{ |\widehat a(l \pm k)-\widehat a(l)| \leq 2dp |\fconnf(l\pm k) - \fconnf(l)| + p|\fpip(l \pm k) - \fpip(l)| \leq (16+M/d) [1-\fconnf(k)]^{1/2}. \label{eq:boot:FR:hat_a_diff_bound}}
Combining~\eqref{eq:boot:FR:hat_a_diff_bound} with Observation~\ref{obs:bootstrap:FR:f4_bound} yields
	\al{\text{(II)} &\leq (16+M/d)^2 [1-\fconnf(k)] \left\vert \frac{1-\lambda_p \fconnf(l)}{1-\widehat a(l)} \right\vert 
					\left\vert \frac{1-\lambda_p \fconnf(l+\sigma k)}{1-\widehat a(l+\sigma k)} \right\vert \fgmu(l) \fgmu(l+\sigma k) \\
			& \leq (144 + M' /d) [1-\fconnf(k)] \fgmu(l) \fgmu(l+\sigma k). }
To bound (III), we want to use Lemma~\ref{lem:bootstrap:Delta_k_Ulam_bound}. We first provide bounds for the three types of quantities arising in the use of the lemma. First, note that $|\widehat a(l)| \leq 4+M/d$. Next, we observe
	\[ \widehat{|a|} (0) - \widehat{|a|}(k) = \sum_x [1-\cos(k\cdot x)] \big\vert 2dp \connf(x) + p\Pi_p(x) \big\vert \leq (4+M/d) [1-\fconnf(k)].\]
The third ingredient we need is Observation~\ref{obs:bootstrap:FR:f4_bound}, which produces $|1-\widehat a(l)|^{-1} \leq 3 \fgmu(l)$. Putting all this together and applying Lemma~\ref{lem:bootstrap:Delta_k_Ulam_bound} gives
	\al{\Delta_k \Big( \frac{1}{1-\widehat a(l)} \Big) & \leq (4+M/d) [1-\fconnf(k)] \bigg( 9 \big[ \fgmu(l-k) + \fgmu(l+k) \big] \fgmu(l) \\
		& \hspace{3cm} + 216 (4+M/d) \fgmu(l-k)\fgmu(l+k) \fgmu(l) [1-\fconnf(l)] \bigg)\\
		& \leq (6912 + M'/d) [1-\fconnf(k)] \Big( \fgmu(l-k) + \fgmu(l+k) \big] \fgmu(l) + \fgmu(l-k)\fgmu(l+k) \Big), }
noting that $\fgmu(l) [1-\fconnf(l)] \leq 2$. In summary, $\text{(I)}+\text{(II)}+\text{(III)} \leq 3 \ulam(k,l),$ which finishes the improvement on $f_3$. 
This finishes the proof of Proposition~\ref{thm:bootstrap:forbidden_region}, and therewith also the proof of Theorem~\ref{thm:main_theorem_triangle_condition}.

\paragraph{Acknowledgment.}  We thank G\"unter Last,  Remco van der Hofstad, and Andrea Schmidbauer for numerous discussions about the lace expansion, which are reflected in the presented paper. 

\bibliography{bibliography}{}
\bibliographystyle{amsplain}

\end{document}